\documentclass[11pt]{amsart}
\usepackage{amsmath}
\usepackage{amsthm}
\usepackage{amssymb}
\usepackage{amsfonts,mathrsfs}
\usepackage{stmaryrd}
\usepackage{amsxtra}  
\usepackage{epsfig}
\usepackage{verbatim}
\usepackage[all]{xy}

\usepackage{dsfont}
\usepackage{color}
\usepackage{mathtools}
\usepackage{tikz}
\usetikzlibrary{fpu}
\usetikzlibrary{matrix}

\theoremstyle{plain}
\newtheorem{theorem}[equation]{Theorem}
\newtheorem{proposition}[equation]{Proposition}
\newtheorem{lemma}[equation]{Lemma}
\newtheorem{corollary}[equation]{Corollary}
\newtheorem{definition}[equation]{Definition}
\newtheorem{conjecture}[equation]{Conjecture}

\usepackage{bm}

\theoremstyle{remark}
\newtheorem{remark}[equation]{Remark}
\numberwithin{equation}{section}

\newtheorem{example}[equation]{Example} 
\newcommand{\dee}{\partial}
\newcommand{\deebar}{\overline\partial}
\newcommand{\w}{\wedge}

\newcommand{\dual}{{\sf dual}}

\newcommand{\ph}[1]{\tau(#1)}
\usepackage{enumerate}
\newenvironment{enum}{  
\begin{enumerate}[\upshape(\arabic{section}.\arabic{equation}a)] }  
{  \end{enumerate}   }
\newcommand{\itemref}[2] {{\upshape(\ref{#1}\ref{#2})}}
\newcommand{\chd}{W}

\renewcommand{\hat}{\widehat}

\usepackage{scalerel,stackengine}
\stackMath
\newcommand\reallywidehat[1]{%
\savestack{\tmpbox}{\stretchto{%
  \scaleto{%
    \scalerel*[\widthof{\ensuremath{#1}}]{\kern.1pt\mathchar"0362\kern.1pt}%
    {\rule{0ex}{\textheight}}%WIDTH-LIMITED CIRCUMFLEX
  }{\textheight}% 
}{2.4ex}}%
\stackon[-6.9pt]{#1}{\tmpbox}%
}
\newcommand{\dbar}{\bar \partial}

\newcommand{\im}{\text{Im}}

\DeclareMathOperator{\re}{Re}

\def\norm#1{\left\Vert#1\right\Vert}

\newcommand{\cf}{{\mathcal F}}
\newcommand{\cg}{{\mathcal G}}

\newcommand{\co}{{\mathcal O}}

\newcommand{\cs}{{\mathcal S}}
\newcommand{\ct}{{\mathcal T}}

\newcommand{\sm}{{\mathscr M}}

\newcommand{\C}{{\mathbb C}}

\newcommand{\R}{{\mathbb R}}

\begin{document}

\title[Leray Transform]{The Leray Transform: factorization, dual $CR$ structures, and model hypersurfaces in $\C\mathbb{P}^2$}
\author{David E. Barrett \& Luke D. Edholm  }
%\subjclass[2010]{32W05}
\begin{abstract}
We compute the exact norms of the Leray transforms for a family $\cs_{\beta}$ of unbounded hypersurfaces in two complex dimensions.  The $\cs_{\beta}$ generalize the Heisenberg group, and provide  local projective approximations to any smooth, strongly $\C$-convex hypersurface $\cs$ to two orders of tangency.    This work is then examined in the context of projective dual $CR$-structures and the corresponding pair of canonical dual Hardy spaces associated to $\cs$, leading to a universal description of the Leray transform and a factorization of the  transform through orthogonal projection onto the conjugate dual Hardy space.
\end{abstract}
%\date{\today}
\thanks{The first author was supported in part by NSF grant number DMS-1500142}
\address{Department of Mathematics\\University of Michigan, Ann Arbor, MI, 48109-1043 USA}
\email{barrett@umich.edu}
\email{edholm@umich.edu}

\maketitle

%%%%%%%%%%%%%%%%%%%%%%%%%%%%%%%% INTRODUCTION %%%%%%%%%%%%%%%%%%%%%%%%%%%%%%%%%%%%%%%%%%%%
\section{Introduction}\label{S:Introduction}

This paper is the first in a series aimed toward better understanding the Leray transform on smooth, strongly $\C$-convex hypersurfaces.    Here, the initial focus is on the following family of models.  For $0 \le \beta <1$, define the hypersurface 
\begin{equation}\label{D:ModelHypersurfaces}
\cs_{\beta} := \left\{ (\zeta_1,\zeta_2)\in \C^2\colon \text{Im}(\zeta_2) = \left|\zeta_1\right|^2 + \beta \re(\zeta_1^2)\right\},
\end{equation}
along with the (unbounded) domain lying on its $\C$-convex side
\begin{equation}\label{D:ModelDomains}
\Omega_{\beta} := \left\{ (z_1,z_2)\in \C^2\colon \text{Im}(z_2) > \left|z_1\right|^2 + \beta\,\text{Re}(z_1^2)\right\}.
\end{equation}

Let $\cs \subset \C\mathbb{P}^n$ be a strongly $\C$-convex hypersurface bounding a domain $\Omega$ on its $\C$-convex side.  The Leray transform $\bm{L}_{\cs}$ (see Remark \ref{R:LerayTerminology} regarding terminology) is a member of the Cauchy-Fantappi\`e class of integral operators re-capturing key properties of the familiar one-variable Cauchy transform.  Its applications include analysis of the Hardy space on the domain $\Omega$:  if $\sigma$ is a measure on $\cs$ and $\bm{L}_{\cs}$ maps $L^2(\cs,\sigma) \to L^2(\cs,\sigma)$ boundedly, 
then the transform identity shows  how functions in the Hardy space $H^2(\cs,\sigma)$ are built from certain rational functions.

As is typical in Hardy space constructions, care must be taken to specify the measure $\sigma$, especially when $\cs$ is unbounded.  For the $\cs_{\beta}$ defined above, the natural measure arising from the Leray transform is $\sigma = dx_1 \wedge dy_1 \wedge dx_2$, where $(x_1,y_1,x_2,y_2)$ are the usual affine coordinates on $\R^4 \cong \C^2$.  This choice of $\sigma$ is a constant multiple of Fefferman hypersurface measure (see \cite{Bar06,Bar16,Fef79,Gup17}) on $\cs_{\beta}$, and consequently has many desirable invariance properties.  (Note that $\sigma$ is not comparable to the Euclidean surface measure on $\cs_{\beta}$.)

The main computation in the first half of the paper is the following:

\begin{theorem}\label{T:ModelNorm}
Let $\bm{L}_{\beta}$ denote the Leray transform of $\cs_{\beta}$, and $\sigma =  dx_1 \wedge dy_1 \wedge dx_2$.  Then $\bm{L}_{\beta}:L^2(\cs_{\beta},\sigma) \to L^2(\cs_{\beta},\sigma)$ is a bounded operator with norm
\begin{equation}\label{E:ModelNorm}
\norm{\bm{L}_{\beta}}= \frac{1}{\sqrt[4]{1-\beta^2}}.
\end{equation}
\end{theorem} 

It is rare to compute exact norms of operators, so this result is interesting in its own right.  But this computation also has significant import since the $\cs_{\beta}$ serve as models in local geometric considerations.  Given a smooth, strongly $\C$-convex hypersurface $\cs \subset \C\mathbb{P}^2$, the $\cs_{\beta}$ can be used to locally approximate $\cs$ to two orders of tangency.  Indeed for each fixed $\zeta^*\in \cs$, there is an automorphism of $\C\mathbb{P}^2$ -- see equation \eqref{E:ProjTrans} -- 
moving $\zeta^*$ to the origin such that the degree two Taylor expansion of $\cs$ after this coordinate change is given by a unique $\cs_{\beta(\zeta^*)}$, $0\le\beta(\zeta^*)<1$.  The precise formulation of this statement is given in Section \ref{S:Background}.  

The mapping properties of the one-dimensional Cauchy transform are well-studied.  In \cite{KerSte78a}, Kerzman and Stein relate the Cauchy transform, $\bm{C}$, on a smooth bounded domain $\Omega \subset \C$ to the Szeg\H o projection and show that the two operators coincide if and only if $\Omega$ is a   disc.  It follows that $\norm{\bm{C}}_2\ge1$ with equality if and only if $\Omega$ is a disc.  However, it can also be extracted from this work that the {\em essential} $L^2$-norm $\norm{\bm{C}}_{\sf{ess}} = 1$ for every smooth $\Omega \subset \C$.  (Recall that the essential norm  -- see \cite{CowMacBook95} -- measures the distance to the set of compact operators; it often occurs in localized analysis.)  See \cite{BarBol07, Bol07, Bol06} for related results about $\bm{C}$.

Much recent work has been done to understand $\bm{L}_{\cs}$ in dimensions two and higher.  In \cite{BarLan09}, the first author and Lanzani study the essential spectrum of this operator on a class of Reinhardt domains in $\C^2$, and their results are in the same vein as Theorem \ref{T:ModelNorm}.  Lanzani and Stein have written a series of recent articles exploring many aspects of this operator.  (See \cite{LanSte13, LanSte14, LanSte17a, LanSte17c, LanSte17b}.)  In \cite{LanSte14}, it is shown that the Leray transform of any {\em bounded} strongly $\C$-convex hypersurface $\cs$ that is at least $C^{1,1}$ smooth maps $L^p(\cs)$ to itself for all $1<p<\infty$.  Counterexamples to the $L^2$-boundedness of $\bm{L}_{\cs}$ exist if either the {\em strong} $\C$-convexity hypothesis or the $C^{1,1}$ smoothness hypothesis is dropped.  See \cite{BarLan09} and \cite{LanSte17b} for more information.  We note that the $\cs_{\beta}$ are unbounded and fail to be even $C^{1}$ at infinity (see Remark \ref{R:Non-C1}).  

In \cite{Bar16}, the first author shows that the $L^2$-norm of $\bm{L}_{\cs}$ measures the effectiveness of pairing two natural Hardy spaces associated to a hypersurface $\cs$.  This is analogous to the role played by the Cauchy transform in an $L^2$-pairing of functions holomorphic inside a Jordan curve with functions holomorphic outside the curve.  (See David's appendix in Meyer's monograph \cite{Meyer82}.)  Recalling the Lanzani-Stein result in the previous paragraph, insight into the interaction of these two Hardy spaces is gained by investigating the norm of $\bm{L}_{\cs}$.

%The local nature of strong $\C$-convex hypersurfaces, along with desirable transformation properties of the Leray operator under projective automorphisms -- see Remark \ref{R:AffineTransforms} -- has led to the following conjecture on its essential $L^2$-norm:  

\begin{conjecture}\label{T:EssentialNormLowerBound}
Let $\cs \subset \C\mathbb{P}^2$ be a smooth bounded, strongly $\C$-convex hypersurface and $\bm{L}_{\cs}$ denote its Leray transform.  The essential $L^2$-norm is given by
\begin{equation}\label{E:LerayEssentialConjecture}
\norm{\bm{L}_{\cs}}_{\sf{ess}} = \sup_{\zeta\in \cs }\frac{1}{\sqrt[4]{1-\beta(\zeta)^2}}.
\end{equation}
\end{conjecture}

We postpone further discussion of this conjecture to a subsequent work in this series, other than a quick run through of evidence for why we believe it to be true.  First, note the similarity between this statement and the situation for the Cauchy transform as described above.  A disc of varying radius will osculate any smooth, bounded domain $\Omega$ in the plane.  The essential norm of an operator frequently arises in local considerations, and the fact that $\norm{\bm{C}_{D}}_2=1$ on every disc $D$ leads to the corresponding result for the essential $L^2$-norm of the Cauchy transform on $\Omega$.  In two complex dimensions the osculating hypersurfaces giving local approximation are the $\cs_{\beta}$, so Theorem \ref{T:ModelNorm} suggests the form of this conjecture.  Additionally, it is shown in \cite{BarLan09} that Conjecture \ref{T:EssentialNormLowerBound} holds for all smooth bounded, strongly convex Reinhardt domains.  

After proving Theorem \ref{T:ModelNorm}, we continue to develop the general theory with the use of projective dual coordinates.  These coordinates depend on the choice of a matrix $M$, corresponding to a particular affinization of projective space.  The Leray transform is shown to be given by a universal formula involving projective dual coordinates in Proposition \ref{P:GenMLeray}.  On every bounded $\C$-convex hypersurface enclosing the origin, this formula resembles the form of the Szeg\H o projection of the unit ball -- compare the formula \eqref{E:LerM1} with \eqref{E:SzegoSphere}.  On every unbounded $\C$-convex hypersurface which can be written as a graph of a smooth function over $\C\times\R$, this formula resembles the form of the Szeg\H o projection of the Siegel upper half space -- compare  \eqref{E:LerM2} with \eqref{E:Szego_S_0}.  This general procedure leads to the identification of a preferred measure on $\cs$ in \eqref{E:NuDef}.  It should be emphasized that while $\bm{L}_{\cs}$ only coincides with the Szeg\H o projection in  special cases (affine images of the unit ball or Siegel upper half space \cite{Bol08}), the resemblance of the formulas suggests a relationship between the two operators that should be further explored.

The dual coordinates may be used to induce a secondary $CR$ structure on $\cs$, called the {\em projective dual} $CR$ structure.  This is carried out in Section \ref{SS:ProjDualCR}.  From the two $CR$ structures and the preferred measure mentioned above, we obtain a pair of Hardy spaces on $\cs$, denoted $H^2(\cs)$ and $H^{2}_\dual(\cs)$.  The restriction $\bm{Q}_{\cs}$ of $\bm{L}_{\cs}$ to 
$\overline{H^2_\dual(\cs)}$ provides an explicit invertible map $\overline{H^2_\dual(\cs)}\to H^2(\cs)$, and the full map $\bm{L}_{\cs}$ factors through $\bm{Q}_{\cs}$.  (Most of the objects mentioned just above depend on the choice of $M$ -- with transparent transformation laws -- but the dual $CR$ structure is independent of $M$.)

The paper is structured as follows.  Section \ref{S:Background} collects notation, definitions and necessary background material.  It also motivates our problem by recalling results from one complex variable.  Section \ref{S:LerayTransform} begins with the analysis of  $\bm{L}_{\beta}$ via one Fourier transform together with size estimates.  We then use certain projective automorphisms of $\cs_{\beta}$ to obtain a new parametrization of $\bm{L}_{\beta}$ in Section \ref{SS:ReparametrizingTheKernel}, leading to the use of a second Fourier transform in the proof of Theorem \ref{T:ModelNorm} in Section \ref{SS:MainProof}.  The projective dual coordinates and related constructions are developed in Section \ref{S:DualCR}.  In Sections \ref{S:FactorLS} and \ref{S:DualHardy} we explain the factorization results for $\bm{L}_{\cs}$, with special attention drawn to the case of $\bm{L}_{\beta}$.   Appendix \ref{A:LbetaConvegence} contains  further analysis  of automorphisms of $\cs_{\beta}$ leading to basic $L^2$-estimates related to $\bm{L}_{\beta}$.

The authors would like to acknowledge the two anonymous referees for many useful suggestions to improve this paper.

%%%%%%%%%%%%%%%%%%%%%%%%%% %%%%%%%%%%%%%%%%%%%%%%%%%% %%%%%%%%%%%%%%%%%%%%%%%%%% %%%%%%%%%%%%%%%%%%%%%%%%%% %%%%%%%%%%%%%%%%%%%%%%%%%% %%%%%%%%%%%%%%%%%%%%%%%%%% %%%%%%%%%%%%%%%%%%%%%%%%%% %%%%%%%%%%%%%%%%%%%%%%%%%% %%%%%%%%%%%%%%%%%%%%%%%%%% %%%%%%%%%%%%%%%%%%%%%%%%%% %%%%%%%%%%%%%%%%%%%%%%%%%% %%%%%%%%%%%%%%%%%%%%%%%%%% %%%%%%%%%%%%%%%%%%%%%%%%%% %%%%%%%%%%%%%%%%%%%%%%%%%% %%%%%%%%%%%%%%%%%%%%%%%%%% %%%%%%%%%%%%%%%%%%%%%%%%%% %%%%%%%%%%%%%%%%%%%%%%%%%% %%%%%%%%%%%%%%%%%%%%%%%%%% %%%%%%%%%%%%%%%%%%%%%%%%%% %%%%%%%%%%%%%%%%%%%%%%%%%% 
%%%%%%%%%%%%%%%%%%%%%%%%%% Section 2: Background %%%%%%%%%%%%%%%%%%%%%%%%%%%%%%%%%%%%%%%%%%%%%%%%

\section{Background}\label{S:Background}

One motivating factor for the study of the Leray transform is the desire for a higher dimensional analogue of the Cauchy transform in one complex variable.  For any smooth bounded domain $\Omega \subset \C$, recall that the Cauchy transform of a function $f$ defined on the boundary $b\Omega$ is given by  
\begin{equation}\label{E:CauchyTransformFormula}
\bm{C}f(z):= \frac{1}{2\pi i}\int_{b\Omega}\frac{f(\zeta)}{\zeta-z} \, d\zeta.
\end{equation}
Denoting the Cauchy kernel by $C(z,\zeta) = \frac{1}{2 \pi i (\zeta-z)} \, d\zeta$, we highlight three essential properties:  
\begin{itemize}
\item[(i)] {\bf Reproduction.}  Given $z\in\Omega$ and $f\in \co(\overline{\Omega})$, we have $\bm{C}f(z) = f(z)$. 

\item[(ii)] {\bf Holomorphicity.} For each fixed $\zeta\in b\Omega$, the kernel $C(z,\zeta)$ is holomorphic in $z$, and thus can be used to construct holomorphic functions. 

\item[(iii)] {\bf Domain Independence.} The kernel $C(z,\zeta)$ is independent of $\Omega$, in that no explicit reference to a defining function of $\Omega$ is made in the coefficient function of $d\zeta$.
\end{itemize}

Unfortunately, these three properties never simultaneously hold in this form for a kernel-operator pair in higher dimensions.  See \cite{LanSte13} for an excellent survey detailing these matters.  We shall insist on keeping the reproduction property (i), and will look for higher dimensional successors of the Cauchy transform by dropping one of the other conditions. The Bochner-Martinelli formula  satisfies (i) and (iii), but fails to be holomorphic in the parameter $z$ (see \cite{Krantz_scv_book,Range86}).  Though useful in many respects, this formula cannot be used to create holomorphic functions from more general boundary data.

Restricting to strongly $\C$-convex hypersurfaces $\cs\subset\C\mathbb{P}^n$, we are able to define the Leray kernel-operator pair which satisfies properties (i) and (ii).  As above, property (ii) allows construction of holomorphic functions from rather general boundary data.  And while the Leray kernel is  domain dependent when $\cs$ is viewed as a subset of $\C\mathbb{P}^n$, a universal description involving projective dual variables is presented in Section \ref{S:DualCR} (see Remark \ref{R:DomainIndependenceLeray}).

%\noindent{\bf ***Should we talk about the norm of the one dimensional Cauchy transform here or elsewhere?  Definitely want to talk about how the norm is 1 if and only if we are integrating over a circle.}

%\noindent{\bf  ***Also talk about the pairing between inside and outside holomorphic functions. Then it would be appropriate to bring up the CR-dual CR pairing and how the norm of the Leray operator measures the efficiency of this pairing.}

\subsection{Strong $\C$-convexity}\label{SS:StrongCConvexity}

% ***Important stuff like what we mean by $L^2$ and what the measure is, different surface measures, and especially the definition of essential norm.

%Hypersurface Bolt papers \cite{Bol08}, \cite{Bol10}

%\subsection{$\C$ convex hypersurfaces}\label{SS:CConvexHypersurfaces}
%{\bf ***Should we restrict ourselves to $\C^2$?}  

%\noindent{\bf *** Strong $\C$-convexity in terms of growth order $\ge C|z-\zeta|^2 $}

%\noindent{\bf *** Locally projectively equivalent to strongly convex}

%\noindent{\bf *** Positivity of 2nd fundamental form}

%\noindent{\bf  *** Why this condition is natural to view in projective space. }

Underlying the Leray transform is the geometric notion of $\C$-convexity.  There are several equivalent definitions (see \cite{ScandBook04}), but we focus here on a differential condition along $\cs$.  We first note the following proposition and refer to Section 5.2 in \cite{Bar16} for its proof.

\begin{proposition}\label{P:StrongCConvexity}
Let $\cs$ be a smooth, strongly pseudoconvex real hypersurface in $\C\mathbb{P}^n$ and let $p\in\cs$.  Then there is an automorphism of $\C\mathbb{P}^n$ moving $p$ to $0$ so that the transformed $\cs$ takes the following form near $0$.
\begin{equation*}
\mathrm{Im}(\zeta_n) = \sum_{j,k=1}^{n-1} \alpha_{j,k}\zeta_j\bar{\zeta}_k + \mathrm{Re}\left(\sum_{j,k=1}^{n-1} \beta_{j,k} \zeta_j \zeta_k \right) + c\, \mathrm{Re}(\zeta_n)^2 + O(\norm{(\zeta_1,\dots,\zeta_{n-1},\mathrm{Re}(\zeta_n))}^3).
\end{equation*}
The real constant $c$ may be set arbitrarily, but the sum $\sum \alpha_{j,k}\zeta_j\bar{\zeta}_k + \mathrm{Re}\left(\sum\beta_{j,k} \zeta_j \zeta_k \right)$ is determined up to a scalar multiple and a $\C$-linear change of coordinates in $(\zeta_1, \dots, \zeta_{n-1})$.
\end{proposition}

We now use this normal form to set up our definition.

\begin{definition}\label{DStrongCConvexity}
Let $\cs$ be connected and strongly pseudoconvex in $\C\mathbb{P}^n$ and suppose (after a projective automorphism) that $\cs$ is given in the form of Proposition \ref{P:StrongCConvexity} above.  We say that $\cs$ is strongly $\C$-convex if and only if
\begin{equation}\label{E:PositiveDefiniteForm1}
\sum_{j,k=1}^{n-1} \alpha_{j,k}\zeta_j\bar{\zeta}_k + \mathrm{Re}\left(\sum_{j,k=1}^{n-1} \beta_{j,k} \zeta_j \zeta_k \right)
\end{equation}
is positive definite on $(\zeta_1,\dots,\zeta_{n-1})$.  This is equivalent to saying that the complex tangent hyperplane $\C^{n-1}\times \{0\}$ has minimal order of contact with $\cs$.
\end{definition}

Because the real constant $c$ in Proposition \ref{P:StrongCConvexity} may be chosen arbitrarily, set $c=0$ and diagonalize the form to 
\begin{equation}\label{E:PositiveDefiniteForm2}
\mathrm{Im}(\zeta_n)=\sum_{j=1}^{n-1}\alpha_j|\zeta_j|^2 + \mathrm{Re}\left(\sum_{j=1}^{n-1}\beta_j \zeta_j^2 \right) + O(\norm{(\zeta_1,\dots,\zeta_{n-1},\mathrm{Re}(\zeta_n))}^3),
\end{equation}
with each $\beta_j\ge0$.  Strong $\C$-convexity at the origin is equivalent to saying that each $\beta_j < \alpha_j$.  When $n=2$, we may set $\alpha_1=1$ which leads to our hypersurfaces
$$
\cs_{\beta} := \left\{ (\zeta_1,\zeta_2)\in \C^2: \text{Im}(\zeta_2) = \left|\zeta_1\right|^2 + \beta\, \rm{Re}(\zeta_1^2)\right\}.
$$
\begin{remark}\label{R:StrongCConvex => StrongPseudoconvex}
Every strongly $\C$-convex hypersurface is (by our definition) strongly pseudoconvex, and osculating biholomorphic images of the sphere  (or of $\cs_{0}$) are a useful tool in the  study  of strongly pseudoconvex hypersurfaces.  The Leray transform has good transformation properties under automorphisms of $\C\mathbb{P}^2$ (but not under general biholomorphic maps); working with this smaller set of mappings we need a larger set of models, and the projective images of the $\cs_{\beta}$ suitably osculate any strongly $\C$-convex hypersurface. $\lozenge$
\end{remark}

%Alternatively, we can think of $\cs_{\beta}$ as a hypersurface in $\R^4$:
%\begin{equation}\label{E:ModelHypersurfaceReal}
%\cs_{\beta} = \left\{ (x_1,y_1,x_2,y_2)\in \R^4: y_2 = (1+\beta)x_1^2 + (1-\beta)y_1^2\right\}.
%\end{equation}

\subsection{The Leray transform}\label{SS:LerayTransform}
Let $\cs$ be a compact strongly $\C$-convex $C^2$ hypersurface in $\C^n$, and let $\Omega$ be the bounded domain  with boundary $\cs$.  If $f$ is a function on $\cs$, the Leray transform maps $f$ to a holomorphic function on $\Omega$ whenever the following integral makes sense.
\begin{align}
&\bm{L}_{\cs}f(z) := \int_{\cs} f(\zeta) \, \mathscr{L}_{\cs}(z,\zeta), \label{E:LerayIntegralFormula} \\
&\mathscr{L}_{\cs}(z,\zeta) := \frac{1}{(2\pi i)^n}  \frac{\partial\rho(\zeta) \wedge (\bar \partial\partial\rho(\zeta))^{n-1}}{\left< \partial\rho(\zeta),(\zeta - z) \right>^n}. \label{E:LerayKernel}
\end{align}
Note that the {\em Leray kernel} $\mathscr{L}_{\cs}$ is a form of bi-degree $(n,n-1)$.  Here $\rho$ is a defining function for $\cs$, and $\left<\cdot,\cdot \right>$ is the natural {\em bilinear} pairing between $(1,0)$-forms and vectors.  This definition is independent of the choice of defining function.  See chapter IV of \cite{Range86} for more information. 

Formula \eqref{E:LerayIntegralFormula} can actually be defined for a more general class of hypersurfaces. If $\cs$ is bounded and has the property that all complex tangent hyperplanes never intersect $\Omega$, the integral is defined for all $z \in \Omega$.  Such hypersurfaces are simply called $\C$-convex.  (See Proposition 2.5.9 in \cite{ScandBook04}.  In the context of smooth connected hypersurfaces the notion of $\C$-convexity coincides with the notion of $\C$-linear convexity, but these notions diverge in more general situations -- see Chapter 2 in \cite{ScandBook04} for a complete discussion of these matters.)

Now consider a compact strongly $\C$-convex $C^2$ hypersurface $\cs$ in $\C\mathbb{P}^n$.  The complex tangent hyperplane  to $\cs$ at some $\zeta\in\cs$ meets $\cs$ only at $\zeta$
(see Remark 2.5.11 in \cite{ScandBook04}).  Pushing the hyperplane away from $\cs$ by a small multiple of the normal vector to $\cs$ at $\zeta$ we obtain a hyperplane disjoint from $\cs$ (lying on the concave side of $\cs$).  This hyperplane may be sent to the hyperplane at infinity by means of a projective transformation.  Thus $\cs$ is projectively equivalent to a  compact strongly $\C$-convex $C^2$ hypersurface in $\C^n$; Bolt's transformation law (previously mentioned in Remark \ref{R:AffineTransforms} above) can now be used to defined the Leray transform in terms of its euclidean counterpart.

Once it has been established that \eqref{E:LerayIntegralFormula} converges for $z \in \Omega$, it is of interest to think of $\bm{L}_{\cs}$ as a map from functions on $\cs$ to functions on $\Omega$, and to understand the boundary values of $\bm{L}_{\cs}f(z)$.  For $z \in \cs$, \eqref{E:LerayIntegralFormula} is a singular integral and must be interpreted in a suitable way.

For {\em unbounded} $\C$-convex hypersurfaces, it is not a priori clear that \eqref{E:LerayIntegralFormula} converges even for $z \in \Omega$.  However when $\cs = \cs_{\beta}$, we show the following hold:
\begin{itemize}

\item[(a)] Given $f \in L^2(\cs_{\beta}, \sigma)$, $\bm{L}_{\beta}f$ is a holomorphic function on $\Omega_{\beta}$.  (Appendix \ref{SS:L_beta(f)isHolo}.)

\item[(b)]  $\bm{L}_{\beta}$ is a bounded operator from $L^2(\cs_{\beta},\sigma) \to L^2(\cs_{\beta},\sigma)$. (Proposition \ref{P:L2Boundedness(NotSharp)}.)

\item[(c)] $\norm{\bm{L}_{\beta}}_{L^2(\cs_{\beta},\sigma)} = \frac{1}{\sqrt[4]{1-\beta^2}}$. (Theorem \ref{T:ModelNorm}, proved in Section \ref{SS:MainProof}.)

\item[(d)] The Leray transform of $\cs_{\beta}$ is a projection, i.e., $\bm{L}_{\beta} \circ \bm{L}_{\beta} = \bm{L}_{\beta}$. (Corollary \ref{C:LerayIsAProjection}.)  It is {\em not} orthogonal, except when $\beta = 0$, in which case the Leray transform coincides with the Szeg\H o projection.

\item[(e)] The closure of $S_\beta$ in complex projective space fails to be $C^1$.  (Remark \ref{R:Non-C1}.)

\end{itemize}

The measure used above is $\sigma = dx_1 \wedge dy_1 \wedge dx_2$.  We postpone the discussion of the surface measures considered on a more general $\cs$ until Section \ref{S:DualCR}.

\begin{remark}
Equation \eqref{E:LerayIntegralFormula} in dimension one reduces to \eqref{E:CauchyTransformFormula}. $\lozenge$
\end{remark}

\begin{remark}\label{R:LerayTerminology}
$\bm{L}_{\cs}$ is often referred to in the literature as the Cauchy-Leray transform, the Leray-Aizenberg transform or the Cauchy transform for convex domains.  $\lozenge$
\end{remark}

\subsection{The Fourier transform}\label{SS:FourierTransform}
The Fourier transform is a standard tool with varying normalizations.  We include this section to collect basic facts for the version used in this paper. 

Use $\cf$ and $\cf^{-1}$ to denote the Fourier and inverse Fourier transforms on $\R$, respectively, with the following convention:  given $h\in L^1(\R)$,
\begin{equation}\label{D:FourierDef}
\cf h(\xi):=  \int_{-\infty}^{\infty} h(x)e^{-2\pi i x\xi}\, dx,
\ \ \ \ \  \ \
\cf^{-1}h(\xi):=  \int_{-\infty}^{\infty} h(x)e^{2\pi i x\xi}\, dx.
\end{equation}
It is standard to extend $\cf$ and $\cf^{-1}$ to operators on $L^2(\R)$ using the density of $L^1 \cap L^2$ functions in $L^2$.  The of placement of the $2\pi$ in \eqref{D:FourierDef} guarantees that both $\cf$ and $\cf^{-1}$ are isometries of $L^2$, i.e., for $g \in L^2(\R)$, we have Plancherel's identity:
\begin{equation}\label{E:Plancherel}
\norm{g}_2 = \norm{\cf g}_2 = \norm{\cf^{-1} g}_2.
\end{equation}

Equations \eqref{E:MultipleFTransformSubscripts} and \eqref{E:FourierConvolution} below are written using the inverse Fourier transform $\cf^{-1}$ because this operator is used extensively in later sections.  When dealing with functions of more than one variable, we often subscript both the transform and the phase space variables.  Starting with a two variable function $H(x,y)$,
\begin{align}\label{E:MultipleFTransformSubscripts}
\cf^{-1}_{x,y}H(\xi_x,\xi_y) := \int_{-\infty}^{\infty} \int_{-\infty}^{\infty} H(x,y) e^{2\pi i (x\xi_x + y\xi_y)} \, dx\,dy.
\end{align}
Also note that the formulas in \eqref{D:FourierDef} transform convolutions in the following way:
\begin{align}\label{E:FourierConvolution}
\cf^{-1}(F*G) &= \cf^{-1}(F)\cdot \cf^{-1}(G).
\end{align}
Finally, note the following integral which arises frequently in computations below
\begin{align}\label{E:FourierGaussian}
\int_{-\infty}^\infty e^{-\pi x^2} e^{2\pi i x\xi}\,dx = e^{-\pi\xi^2}.
\end{align}

%%%%%%%%%%%%%%%%%%%%%%%%%% %%%%%%%%%%%%%%%%%%%%%%%%%% %%%%%%%%%%%%%%%%%%%%%%%%%% %%%%%%%%%%%%%%%%%%%%%%%%%% %%%%%%%%%%%%%%%%%%%%%%%%%% %%%%%%%%%%%%%%%%%%%%%%%%%% %%%%%%%%%%%%%%%%%%%%%%%%%% %%%%%%%%%%%%%%%%%%%%%%%%%% %%%%%%%%%%%%%%%%%%%%%%%%%% %%%%%%%%%%%%%%%%%%%%%%%%%% %%%%%%%%%%%%%%%%%%%%%%%%%% %%%%%%%%%%%%%%%%%%%%%%%%%% %%%%%%%%%%%%%%%%%%%%%%%%%% %%%%%%%%%%%%%%%%%%%%%%%%%% %%%%%%%%%%%%%%%%%%%%%%%%%% %%%%%%%%%%%%%%%%%%%%%%%%%% %%%%%%%%%%%%%%%%%%%%%%%%%% %%%%%%%%%%%%%%%%%%%%%%%%%% %%%%%%%%%%%%%%%%%%%%%%%%%% %%%%%%%%%%%%%%%%%%%%%%%%%% 
%%%%%%%%%%%%%%%%%%%%%%%%%%%%%%%%%%%%%%%%% SECTION 3:  THE LERAY TRANSFORM %%%%%%%%%%%%%%%%%%%%%%%%%%%%%%%%%%%%%%%%%%%%%%%%%%%%%%%%%%%% 

\section{The Leray transform of $\cs_{\beta}$}\label{S:LerayTransform}

The Lanzani-Stein machinery in \cite{LanSte14} cannot be directly applied to the Leray transform of $\cs_{\beta}$ because these models fail to be $C^{1,1}$ at $\infty$.  (In fact, they are not even $C^1$ at $\infty$; see Remark \ref{R:Non-C1}.)  In this section, we first use a single Fourier transform along with size estimates to establish the $L^2$-boundedness of $\bm{L}_{\beta}$.  However, sharpness is lost in the computation of the exact norm when oscillatory cancellation is ignored.  This is remedied in Section \ref{SS:ReparametrizingTheKernel} when certain projective automorphisms  are used to re-parametrize $\cs_{\beta}$, ultimately leading to the sharp result given in Theorem \ref{T:ModelNorm}.  

The verification that integral \eqref{E:LerayIntegralFormula} converges and defines a holomorphic function on $\Omega_{\beta}$ for all $f\in L^2(\cs_\beta,\sigma)$ is postponed to Appendix \ref{A:LbetaConvegence}.

\subsection{$L^2$-boundedness via Fourier methods and size estimates}\label{SS:L2ViaSizeEstimates}

Starting from the definition of $\cs_{\beta}$ in \eqref{D:ModelHypersurfaces}, choose defining function $\rho(\zeta) = |\zeta_1|^2 + \beta \re{\zeta_1^2}-\im(\zeta_2)$.  The pieces of the Leray kernel in \eqref{E:LerayKernel} are calculated to be
\begin{align}
\frac{1}{(2\pi i)^2} \dee\rho\wedge\dbar\dee\rho(\zeta) = \frac{1}{8\pi^2 i} d\zeta_2\wedge d\bar\zeta_1\wedge d\zeta_1,\\
\left< \partial\rho(\zeta),(\zeta - z) \right> = (\bar\zeta_1 +\beta\zeta_1)(\zeta_1 - z_1)+ \tfrac i2(\zeta_2-z_2),
\end{align}
and the Leray transform \eqref{E:LerayIntegralFormula} takes the form
\begin{equation}\label{E:LerayModelHypersurface3.1}
\bm{L}_{\beta}f(z) = \frac{1}{8\pi^2i} \int_{\cs_{\beta}} \frac{f(\zeta)\, d\zeta_2 \wedge d\bar\zeta_1 \wedge d\zeta_1}{[(\bar\zeta_1 +\beta\zeta_1)(\zeta_1 - z_1)+ \frac i2(\zeta_2-z_2)]^2}.
\end{equation}

We parametrize $\cs_{\beta}$ by $\R^3$, letting $\zeta_j = x_j+iy_j$ and $z_j = u_j+iv_j$.  The denominator of the integrand becomes $\left( C + \frac i2(x_2-u_2) \right)^2$, where 
\begin{align}
C &= \Big((1+\beta)x_1(x_1-u_1) + (1-\beta)y_1(y_1-v_1) \Big) - \frac 12(y_2 - v_2) \label{E:DefCGeneral} \\
& \hspace{10em} +i\Big((1+\beta)x_1(y_1-v_1) - (1-\beta)y_1(x_1-u_1)\Big). \notag
\end{align}
When both $\zeta,z \in \cs_{\beta}$, this simplifies to
\begin{align}
C &= \label{E:ParameterC} \frac{1+\beta}{2}(u_1-x_1)^2 + \frac{1-\beta}{2}(v_1-y_1)^2 \\
& \hspace{5em} +i\Big((1+\beta)x_1(y_1-v_1) - (1-\beta)y_1(x_1-u_1)\Big) \notag.
\end{align}
In particular, note that $\text{Re}(C) \ge 0$, with equality {\em if and only if} both $x_1=u_1$ and $y_1=v_1$.  

It can be checked that $d\zeta_2 \wedge d\bar\zeta_1 \wedge d\zeta_1 = 2i \, dx_1 \wedge dy_1 \wedge dx_2$.  Using \eqref{E:ParameterC}, we see 
\begin{align}
\eqref{E:LerayModelHypersurface3.1} &= \bm{L}_{\beta}f(u_1,v_1,u_2) \notag \\ 
&= \frac{1}{4\pi^2} \int_{\R^3} \frac{f(x_1,y_1,x_2)}{\big(C + \frac i2(x_2-u_2)\big)^2}\, dx_1 \wedge dy_1 \wedge dx_2 \notag \\
&= -\frac{1}{\pi^2} \int_{\R^2} \bigg( \int_{-\infty}^{\infty} \frac{f(x_1,y_1,x_2)}{\big((u_2-x_2)+2iC\big)^2}\, dx_2 \bigg) dx_1 \wedge dy_1 \notag \\
&= -\frac{1}{\pi^2} \int_{\R^2} F*G(u_2) \, dx_1 \wedge dy_1 \label{E:LerayRealConvolution},
\end{align}
where $F(x_2):=f(x_1,y_1,x_2)$ and $G(x_2) := \frac{1}{(x_2+2iC)^2}$.
\smallskip

We now calculate the inverse Fourier transform of the function $G$.

\begin{proposition}\label{P:InverseFourierComputation}
Let $C$ be a complex number with $\mathrm{Re}(C)>0$.  The inverse Fourier transform of $G(x) = \frac{1}{(x+2iC)^2}$ is given by
\begin{align*}
	\cf^{-1}G(\xi) := \int_{-\infty}^{\infty} G(x)e^{2\pi i x \xi} \, dx
	= \begin{dcases}
		0 & \xi \ge 0\\ \\
		4\pi^2\xi e^{4\pi\xi C} & \xi < 0.\\ 
		\end{dcases}
\end{align*}
\end{proposition}
\begin{proof}
The condition $\text{Re}(C)>0$ means the function $H(z) := \frac{e^{2\pi i z \xi}}{(z+2iC)^{2}}$ has a single pole of order $2$ in the lower half plane.  For any $R>0$, the integral
\begin{equation*}
\int_{-R}^R \frac{e^{2\pi i x \xi}}{(x+2iC)^{2}}\,dx
\end{equation*}
can be thought of as a piece of a contour integral around a semicircle with base on the $x$-axis.

When $\xi>0$, consider such a semicircle in the upper half plane traversed counterclockwise.  The radial portion of the integral tends to $0$ as $R\to0$.  On the other hand, this function is holomorphic inside the contour.  Thus Cauchy's theorem implies $\cf^{-1}G(\xi) = 0$ for $\xi>0$.

When $\xi<0$, consider a semicircle in the lower half plane traversed clockwise.  For sufficiently large $R$, this contour encloses the pole of $H$.  As above, the radial portion of this integral tends to $0$ as $R\to0$.  Thus, Cauchy's integral formula shows
\begin{align*}
\lim_{R\to\infty} \int_{-R}^R \frac{e^{2\pi i x \xi}}{(x+2iC)^{2}}\,dx &= -2\pi i \frac{d}{dx}\Big( e^{2\pi i x \xi} \Big)\Bigg|_{x=-2iC} = 4\pi^2 \xi e^{4\pi\xi C}.
\end{align*}
\end{proof}

Equation \eqref{E:LerayRealConvolution} and Proposition \ref{P:InverseFourierComputation} now show that
\begin{equation}\label{E:U2InvFTransform}
\cf^{-1}_{u_2}\bm{L}_{\beta}f(u_1,v_1,\xi_{u_2}) 
	= \begin{dcases}
		0, & \xi_{u_2} \ge 0\\ \\
		-4 \xi_{u_2} \int_{\R^2} \cf_{u_2}^{-1}f(x_1,y_1,\xi_{u_2}) e^{4\pi\xi_{u_2} C} \,dx_1\,dy_1, & \xi_{u_{2}} < 0.\\ 
		\end{dcases}
\end{equation}
For $\xi_{u_2}<0$, equation \eqref{E:ParameterC} and the triangle inequality show
\begin{align}
\left| \cf^{-1}_{u_2}\bm{L}_{\beta}f(x_1,y_1,\xi_{u_2}) \right| &\le \int_{\R^2} \left|\cf_{u_2}^{-1}f(x_1,y_1,\xi_{u_2})\right| K(u_1-x_1,v_1-y_1,\xi_{u_2})\,dx_1 dy_1  \notag \\
&= \left|\cf_{u_2}^{-1}f(\cdot,\cdot,\xi_{u_2})\right| *_{(u_1,v_1)} K(\cdot,\cdot,\xi_{u_2}) \label{E:TriangleInqeq1},
\end{align}
where
\begin{equation*}
K(x,y,\xi_{u_2}) := -4\xi_{u_2}  \text{Exp}\left[ 2\pi\xi_{u_2}(1+\beta)x^2 + 2\pi\xi_{u_2}(1-\beta)y^2 \right].
\end{equation*}

Equation \eqref{E:TriangleInqeq1} is in convolution form because the triangle inequality lets us ignore the oscillatory piece of the exponential.  Now apply a two-dimensional inverse Fourier transform.  Making use of integral \eqref{E:FourierGaussian}, a computation shows
\begin{equation}\label{E:2DFourierTrans}
\cf^{-1}_{u_1,v_1}K(\xi_{u_1,}\xi_{v_1},\xi_{u_2}) = \frac{2}{\sqrt{1-\beta^2}}\text{Exp}\left[{\frac{\pi}{2\xi_{u_2}}}\left(\frac{\xi^2_{u_1}}{1+\beta} + \frac{\xi^2_{v_1}}{1-\beta} \right) \right],
\end{equation}
and it follows that
\begin{equation}\label{E:SupNormCalculation}
\sup \Big\{ \cf^{-1}_{u_1,v_1}K(\xi_{u_1,}\xi_{v_1},\xi_{u_2}) : \xi_{u_1}\in\R,\ \xi_{v_1}\in \R,\ \xi_{u_2}<0 \Big\} = \frac{2}{\sqrt{1-\beta^2}} .
\end{equation}
We now deduce that
\begin{multline}\label{E:LocalConvoToProduct}
\cf_{u_1,v_1}^{-1}\Big[\left|\cf_{u_2}^{-1}f(\cdot,\cdot,\xi_{u_2})\right| *_{(u_1,v_1)} K(\cdot,\cdot,\xi_{u_2}) \Big]  \\
= \cf_{u_1,v_1}^{-1}\left(\left|\cf_{u_2}^{-1}f\right|\right)(\xi_{u_1},\xi_{v_1},\xi_{u_2})  \cdot \frac{2}{\sqrt{1-\beta^2}} \,\text{Exp}\left[{\frac{\pi}{2\xi_{u_2}}}\left(\frac{\xi^2_{u_1}}{1+\beta} + \frac{\xi^2_{v_1}}{1-\beta} \right) \right]. 
\end{multline}
Equations \eqref{E:SupNormCalculation} and \eqref{E:LocalConvoToProduct} give
\begin{multline}\label{E:LocalSupNormEstimate}
\norm{\cf_{u_1,v_1}^{-1}\Big[\left|\cf_{u_2}^{-1}f(\cdot,\cdot,\xi_{u_2})\right| *_{(u_1,v_1)} K(\cdot,\cdot,\xi_{u_2}) \Big]}_{L^2(\R^3,\,\sigma)} \\
\le \frac{2}{\sqrt{1-\beta^2}} \norm{\cf_{u_1,v_1}^{-1}\Big(\left|\cf_{u_2}^{-1}f\right|\Big)(\xi_{u_1},\xi_{v_1},\xi_{u_2})}_{L^2(\R^3,\,\sigma)}.  \qquad\qquad\qquad\qquad
\end{multline}

This is now summarized to confirm item (b) from Section \ref{SS:LerayTransform}:

\begin{proposition}\label{P:L2Boundedness(NotSharp)}
The Leray transform $\bm{L}_{\beta}$ is bounded from $L^2(\cs_{\beta},\sigma) \to L^2(\cs_{\beta},\sigma)$.
\end{proposition}

\begin{proof}
First note that from repeated application of Plancherel's identity \eqref{E:Plancherel}, we have that
\begin{align*}
\norm{f(u_1,v_1,u_2)}_{L^2(\R^3,\,\sigma)} &= \norm{\cf_{u_2}^{-1} f(u_1,v_1,\xi_{u_2})}_{L^2(\R^3,\,\sigma)} \\
&= \norm{|\cf_{u_2}^{-1} f(u_1,v_1,\xi_{u_2})|}_{L^2(\R^3,\,\sigma)} \\
&= \norm{\cf_{u_1,v_1}^{-1}(|\cf_{u_2}^{-1} f|)(\xi_{u_1},\xi_{v_1},\xi_{u_2})}_{L^2(\R^3,\,\sigma)}.
\end{align*}

Since $\cf_{u_2}^{-1}\bm{L}_{\beta}f(u_1,v_1,\xi_{u_2}) = 0$ for $\xi_{u_2}\ge 0$, we only need to consider $\xi_{u_2}<0$ in the norm computations below.  Now,
\begin{align}
\norm{\bm{L}_{\beta}f(u_1,v_1,u_2)}_{L^2(\R^3,\,\sigma)} &= \norm{\cf_{u_2}^{-1}\bm{L}_{\beta}f(u_1,v_1,\xi_{u_2})}_{L^2(\R^3,\,\sigma)} \notag \\
&\le \norm{|\cf_{u_2}^{-1}f(\cdot,\cdot,\xi_{u_2})| *_{(u_1,v_1)} K(\cdot,\cdot,\xi_{u_2})}_{L^2(\R^3,\,\sigma)} \label{E:LocalIneq1}\\
&= \norm{\cf_{u_1,v_1}^{-1}\Big[\left|\cf_{u_2}^{-1}f(\cdot,\cdot,\xi_{u_2})\right| *_{(u_1,v_1)} K(\cdot,\cdot,\xi_{u_2}) \Big]}_{L^2(\R^3,\,\sigma)} \notag \\
&\le \frac{2}{\sqrt{1-\beta^2}} \norm{\cf_{u_1,v_1}^{-1}\Big(\left|\cf_{u_2}^{-1}f\right|\Big)(\xi_{u_1},\xi_{v_1},\xi_{u_2})}_{L^2(\R^3,\,\sigma)} \label{E:LocalIneq2}\\
&= \frac{2}{\sqrt{1-\beta^2}} \norm{f(u_1,v_1,u_2)}_{L^2(\R^3,\,\sigma)}. \notag
\end{align}
This is equivalent to saying
\begin{equation}\label{E:CrudeUpperBound}
\norm{\bm{L}_{\beta}}_{L^2(\cs_{\beta},\,\sigma)} \le \frac{2}{\sqrt{1-\beta^2}}.
\end{equation}
Inequality \eqref{E:LocalIneq1} follows from \eqref{E:TriangleInqeq1} and \eqref{E:LocalIneq2} is just \eqref{E:LocalSupNormEstimate}. 
\end{proof}

\noindent {\em Note.} A version of \eqref{E:CrudeUpperBound} was shown previously to the first author by Jennifer Brooks.

\begin{remark}\label{R:SharpenNorm}
The bound on $\norm{\bm{L}_{\beta}}_2$ given in equation \eqref{E:CrudeUpperBound} is never sharp, given the validity of Theorem \ref{T:ModelNorm}.  When $\beta=0$, we can sharpen the estimate $\norm{\bm{L}_0}_2<2$ simply by observing the Leray transform is identical to the Szeg\H o projection on $\cs_0$ -- see equation \eqref{E:Szego_S_0}.  Consequently, $\norm{\bm{L}_0}_2=1$.  The operators do not coincide for  $\beta \ne 0$.  Note that while $\cs_{\beta}$ is biholomorphically equivalent to $\cs_0$ via the map $(z_1,z_2)\mapsto(z_1,z_2-i\beta z_1^2)$, they are not projectively equivalent.  C.f. Remark \ref{R:AffineTransforms}.  $\lozenge$
\end{remark}

%\com{This automorphism of $\C^2$ is $(w_1,w_2) \mapsto (w_1,w_2+ i \beta z_1^2)}

\begin{remark}\label{R:WeylTransform}
While Proposition \ref{P:L2Boundedness(NotSharp)} is weaker than Theorem \ref{T:ModelNorm}, it is included for two reasons.  First, it is worthwhile to see that the $L^2$-boundedness of $\bm{L}_{\beta}$ follows straightforwardly from the application of a single Fourier transform \eqref{E:U2InvFTransform}, without the use of the second transform facilitated by the re-parametrization of $\cs_{\beta}$ in Section \ref{SS:ParametrizingTheModels} below.  Additionally, it is of great interest to the authors to understand how much sharpness is lost when applying the triangle inequality in \eqref{E:TriangleInqeq1} and similar situations.  For instance, is the sharp norm directly attainable from equation \eqref{E:U2InvFTransform}?  $\cs_{\beta}$ may be identified with the Heisenberg group (the boundary of the Siegel upper half space) when $\beta = 0$.  This allows for the use of the {\em Weyl transform}, which is tailored to handle the {\em twisted convolution} occurring in \eqref{E:U2InvFTransform}.  See chapter 1 of \cite{Folland89} for more information on these topics.  This machinery gives a direct way to show that $\norm{\bm{L}_0}_2=1$ (as opposed to the indirect observation in the preceding remark), but (yet again) does not apply in the $\beta \ne 0$ case. $\lozenge$
\end{remark}

\subsection{Parametrizing $\cs_{\beta}$ with projective automorphisms}\label{SS:ParametrizingTheModels}

The work in this section is inspired by what is known in the case of the Heisenberg group, which corresponds to $\cs_{0}$.   Immediately from its definition in equation \eqref{D:ModelHypersurfaces}, we see that $\cs_{\beta}$ is invariant under translations of the form $(\zeta_1,\zeta_2) \mapsto (\zeta_1, \zeta_2 + s)$, where $s$ is a real number.  We would like to find less trivial automorphisms of this hypersurface.  We seek maps $\phi:\cs_{\beta} \to \cs_{\beta}$ of the form $\phi(\zeta_1,\zeta_2) = (\zeta_1+c, \cdot)$, $c \in \C$.  The second component is now determined.

\begin{proposition}
Let $c\in\C$ and $s\in\R$.  A complex affine map preserving $\cs_{\beta}$ which translates the first coordinate $\zeta_1 \mapsto \zeta_1 +c$ must be of the form
\begin{equation}\label{E:DefinitionPhi_(c,s)}
\phi_{(c,s)}(\zeta_1,\zeta_2) = \big(\zeta_1+c, \zeta_2 + 2i(\bar c + \beta c)\zeta_1 + i\big(|c|^2+\beta\, {\mathrm{Re}}(c^2)\big)+ s \big).
\end{equation}
\end{proposition}
\begin{proof}
Write $\phi_{(c,s)}(\zeta_1,\zeta_2) = (\zeta_1+c,\,\phi_2(\zeta_1,\zeta_2))$.  Since the image of this map lies in $\cs_{\beta}$,
\begin{align*}
\text{Im}(\phi_2(\zeta_1,\zeta_2)) &= |\zeta_1+c \,|^2 + \beta \, \text{Re}\left((\zeta_1+c)^2\right) \\
&= |\zeta_1|^2 + |c|^2 + 2\,\text{Re}(\zeta_1\bar{c}) + \beta \, \text{Re}(\zeta_1^2) + 2\beta \, \text{Re}(\zeta_1 c) + \beta \, \text{Re}(c^2)  \\
&= \text{Im}(\zeta_2) + 2 \, \text{Re}((\bar{c}+\beta c)\zeta_1) + |c|^2 + \beta \, \text{Re}(c^2).
\end{align*}
Since $\phi_2$ is holomorphic, we must have
\begin{align*}
\phi_2(\zeta_1,\zeta_2) = \zeta_2 + 2i(\bar c + \beta c)\zeta_1 + i\big(|c|^2+\beta \, {\re}(c^2)\big)+ s(\zeta_1,\zeta_2),
\end{align*}
where $s(\zeta_1,\zeta_2)$ is a real valued function.  But this implies $s(\zeta_1,\zeta_2)=s$ is constant.
\end{proof}

We now see how these maps compose.

\begin{proposition}
Let $\phi_{(c,s)}: \cs_{\beta} \to \cs_{\beta}$ be defined as in equation \eqref{E:DefinitionPhi_(c,s)} above.  Then composition of these maps gives
\begin{equation}\label{E:AutomorphismComposition}
\phi_{(c_1,s_1)} \circ \phi_{(c_2,s_2)} = \phi_{\left(c_1+c_2, \,s_1+s_2-2 \mathrm{Im}((\bar{c}_1 + \beta c_1)c_2)\right)}.
\end{equation}
\end{proposition}
\begin{proof}
From equation \eqref{E:DefinitionPhi_(c,s)},
\begin{align}
\phi_{(c_1,s_1)} \circ \phi_{(c_2,s_2)}(\zeta) &= \phi_{(c_1,s_1)}\bigg(\zeta_1+c_2 \, , \zeta_2 + 2i(\bar c_2 + \beta c_2)\zeta_1 + i\big(|c_2|^2+\beta\,{\text Re}(c_2^2)\big)+ s_2 \bigg) \notag \\
&= \bigg(\zeta_1+(c_1 +c_2) \, , \zeta_2 + 2i(\bar c_2 + \beta c_2)\zeta_1 + i\big(|c_2|^2+\beta\,{\text Re}(c_2^2)\big)+ s_2 \label{E:AutoCompIntermediate1} \\ 
& \qquad \qquad \qquad + 2i(\bar{c}_1+ \beta c_1)(\zeta_1 +c_2) + i\big(|c_1|^2+\beta\,{\text Re}(c_1^2)\big)+ s_1\bigg). \notag
\end{align}
From here, the {\em second component} in equation \eqref{E:AutoCompIntermediate1} can be written as
\begin{align*}
\zeta_2 + 2i\left(\overline{c_1+c_2} +\beta(c_1+c_2)\right)\zeta_1 &+ i\left(|c_1+c_2|^2 + \beta\,\text{Re}\left((c_1+c_2)^2\right) \right) \\ &+ (s_1+s_2) -2\, \text{Im}\left((\bar{c}_1+\beta c_1)c_2\right).
\end{align*} \end{proof}

\begin{remark}\label{R:AffineTransforms}
Note that the automorphisms described above are affine maps preserving both volume and our distinguished boundary measure $\sigma = dx_1 \wedge dy_1 \wedge dx_2$.  Thus, it is unnecessary to rescale the Leray transform of $\cs_{\beta}$ when the hypersurface is re-parametrized using these maps.  Affine maps form a subgroup of the automorphisms of $\C\mathbb{P}^n$.  In \cite{Bol05}, Bolt proves a transformation law of the Leray kernel under projective automorphisms, when the kernel is expressed in terms of Fefferman hypersurface measure.  We return to this point in Appendix \ref{SS:Aut of Omega_{beta}}  $\lozenge$
\end{remark}

%shows that if $\psi: S_1 \to S_2$ is any projective automorphism, and $\mathscr{L}_1$ and $\mathscr{L}_2$ are the corresponding Leray kernels defined using the Fefferman hypersurface measure (see \cite{Fef79}, \cite{Bar06}, \cite{Gup17}), then 
%\begin{equation}\label{E:LerayProjInvarience}
%\mathscr{L}_1(z,w) = \det[\psi'(z)]^{n/(n+1)}\cdot \mathscr{L}_2(\psi(z),\psi(w))\cdot \overline{\det[\psi'(w)]^{n/(n+1)}}.
%\end{equation}

It will be desirable to consider an abelian subgroup of the group of the automorphisms defined by equation \eqref{E:DefinitionPhi_(c,s)}.  Notice that the term $-2 \,\text{Im}((\bar{c}_1 + \beta c_1)c_2)$ vanishes when both $c_1,c_2 \in \R$.  This immediately implies the following corollary, where we've changed all instances of $c$ to $r$ to emphasize this parameter is now restricted to real values. 

\begin{corollary}
The collection of automorphisms $\cg := \left\{ \phi_{(r,s)} : r,s\in\R \right\}$ is a closed abelian subgroup of $\text{Aut} (\cs_{\beta})$.  In fact, given two maps in this subgroup,
\begin{equation}
\phi_{(r_1,s_1)} \circ \phi_{(r_2,s_2)} = \phi_{(r_1+r_2,s_1+s_2)}.
\end{equation}
\end{corollary}

We now use the action of $\cg$ on a one-dimensional curve $\gamma$ lying in $\cs_{\beta}$ to re-parametrize this hypersurface.

\begin{theorem}\label{T:CSTParametrization}
Consider the curve $\gamma: \R \to \cs_{\beta}$ given by $\gamma(t) = \left(i t, i(1-\beta)t^2 \right)$. The action of the group $\cg = \left\{ \phi_{(r,s)} : r,s\in\R \right\}$ on the image of $\gamma$ gives a parametrization of $\cs_{\beta}$, i.e., for each $\zeta \in \cs_{\beta}$, there is a unique $(r,s,t) \in \R^3$ such that $\zeta = \phi_{(r,s)}(\gamma(t))$.
\end{theorem}
\begin{proof}
From equation \eqref{E:DefinitionPhi_(c,s)} we can check that 
\begin{align}\label{E:CSTParametrizationDef}
\phi_{(r,s)}(\gamma(t)) = \left(r +it, \, s - 2(1+\beta)r t + i\left[(1+\beta)r^2 +(1-\beta)t^2\right]\right).
\end{align}
The first coordinate can attain any complex number by specifying the parameters $r$ and $t$.  Once these values are decided, we can appropriately choose $s$ to adjust the real part of the second coordinate.
\end{proof}

\subsection{Re-parametrizing the Leray kernel}\label{SS:ReparametrizingTheKernel}  We now make use of the automorphisms in the previous subsection.  Recall that
\begin{equation}\label{E:LerayModelHypersurface4.1}
\bm{L}_{\beta}f(z) = \frac{1}{8\pi^2i} \int_{\cs_{\beta}} \frac{f(\zeta)\, d\zeta_2 \wedge d\bar\zeta_1 \wedge d\zeta_1}{[(\bar\zeta_1 +\beta\zeta_1)(\zeta_1 - z_1)+ \frac i2(\zeta_2-z_2)]^2}.
\end{equation}
In order to circumvent the loss of sharpness on the norm coming from application of the triangle inequality, re-write \eqref{E:LerayModelHypersurface4.1} in terms of the parameterization described in Theorem \ref{T:CSTParametrization}.  For notational purposes, we use $(r_z,s_z,t_z)$ and $(r_{\zeta},s_{\zeta},t_{\zeta})$ to correspond to the respective $z$ and $\zeta$ variables.  In other words,
\begin{align*}
z &= \left(r_z +it_z, \, s_z - 2(1+\beta)r_z t_z + i\left[(1+\beta)r_z^2 +(1-\beta)t_z^2\right]\right), \\
\zeta &= \left(r_{\zeta} +it_{\zeta}, \, s_{\zeta} - 2(1+\beta)r_{\zeta} t_{\zeta} + i\left[(1+\beta)r_{\zeta}^2 +(1-\beta)t_{\zeta}^2\right]\right).
\end{align*}

The wedge product of differentials appearing in the numerator of equation \eqref{E:LerayModelHypersurface4.1} can be written as
\begin{equation}\label{E:DSDCDT}
d\zeta_2 \wedge d\bar\zeta_1 \wedge d\zeta_1 = 2i \, ds_{\zeta}\wedge dr_{\zeta}\wedge dt_{\zeta}.
\end{equation}
The major advantage of using this parametrization is seen when considering the denominator of the integrand in \eqref{E:LerayModelHypersurface4.1}.  After the change of variables,
\begin{align}
(\bar\zeta_1 +\beta\zeta_1)(\zeta_1 - z_1) &= \left((1+\beta)r_{\zeta} - i (1-\beta)t_{\zeta}\right)\left((r_{\zeta}-r_z) + i(t_{\zeta}-t_z)\right) \notag\\
&= (1+\beta)r_{\zeta}(r_{\zeta}-r_z) + (1-\beta)t_{\zeta}(t_{\zeta}-t_z) \label{E:DenomPiece1} \\
& \qquad \qquad \qquad + i\left[(1+\beta)r_{\zeta}(t_{\zeta}-t_z) - (1-\beta)t_{\zeta}(r_{\zeta}-r_z) \right], \notag
\end{align}
and
\begin{align}
\frac i2(\zeta_2-z_2) &=  -\frac12 \left[(1+\beta) (r_{\zeta}^2-r_z^2) + (1-\beta)(t_{\zeta}^2-t_z^2)\right] \label{E:DenomPiece2} \\ & \qquad \qquad \qquad +\frac i2 \left[ (s_{\zeta} - s_z) - 2(1+\beta)(r_{\zeta}t_{\zeta}-r_z t_z) \right].\notag
\end{align}
Putting the pieces together, we obtain
\begin{align}
\eqref{E:DenomPiece1} + \eqref{E:DenomPiece2} &= \frac{1+\beta}{2}(r_z-r_{\zeta})^2 + i(r_z-r_{\zeta})(t_z+t_{\zeta} + \beta(t_z-t_{\zeta})) \label{E:DenomPiece1+2} \\ & \qquad \qquad \qquad + \frac{1-\beta}{2}(t_z-t_{\zeta})^2 - \frac i2(s_z-s_{\zeta}) \notag \\ 
&:= A - \frac i2 (s_z - s_{\zeta}) \notag ,
\end{align}
where we've collected all terms not involving $(s_z - s_{\zeta})$ into the temporary label
\begin{equation}\label{D:DefLabelA}
A = \frac{1+\beta}{2}(r_z-r_{\zeta})^2 + \frac{1-\beta}{2}(t_z-t_{\zeta})^2 + i(r_z-r_{\zeta})(t_z+t_{\zeta} + \beta(t_z-t_{\zeta})).
\end{equation}

The fact that the term involving the $s$ variables appears in convolution form suggests the use of an (inverse) Fourier transform in this variable.  Note that all terms involving the $r$ variables also appear in convolution form.  But rather than performing a two-dimensional inverse Fourier transform, we make use of the computations in Section \ref{SS:L2ViaSizeEstimates}.  From \eqref{E:LerayModelHypersurface4.1}, \eqref{E:DSDCDT} and \eqref{E:DenomPiece1+2},
\begin{align}
\bm{L}_{\beta}f(r_z,s_z,t_z) &= \frac{1}{4\pi^2} \int_{\R^3} \frac{f\left(r_{\zeta},s_{\zeta},t_{\zeta}\right)}{\left(A - \frac{i}{2}(s_z-s_{\zeta}) \right)^2} \, ds_{\zeta}\wedge dr_{\zeta}\wedge dt_{\zeta}  \label{E:LerayTransAutoReParam} \\
&= -\frac{1}{\pi^2} \int_{\R^2} \left( \int_{-\infty}^{\infty} \frac{f\left(r_{\zeta},s_{\zeta},t_{\zeta}\right)}{\left((s_z-s_{\zeta})+2iA\right)^2} \, ds_{\zeta} \right)\, dr_{\zeta} \wedge dt_{\zeta} \notag \\
&= -\frac{1}{\pi^2} \int_{\R^2} F*G(s_z) \, dr_{\zeta} \wedge dt_{\zeta}, \label{E:ReParam1stConvolution}
\end{align}
where $F(s):= f(r_{\zeta},s,t_{\zeta})$ and $G(s):= \frac{1}{(s+2iA)^2}$.
\smallskip

At this point we compare \eqref{E:LerayRealConvolution} to \eqref{E:ReParam1stConvolution} and see that they are identical except that $A = \eqref{D:DefLabelA}$ and $C = \eqref{E:ParameterC}$ are different.

\begin{remark}\label{R:FourierTransformingSingularIntegral}
A glance at equation \eqref{D:DefLabelA} shows $\text{Re}(A) \ge 0$, with equality if and only if both $r_z = r_{\zeta}$ and $t_z = t_{\zeta}$. When $\text{Re}(A)=0$, $F*G$ is defined by a singular integral that should be interpreted in a principal-value sense.  This is reminiscent of the Cauchy transform and relates to topics like the Plemelj jump formula as described in \cite{Mus92}.  We can sidestep this issue, however, since $\text{Re}(A)$ vanishes only on a set of three-dimensional measure zero. $\lozenge$
\end{remark}

$\re{(A)}>0$ almost everywhere and for all such $A$, Proposition \ref{P:InverseFourierComputation} says the inverse Fourier transform of $G(s) = \frac{1}{(s+2iA)^2}$ is given by
\begin{align}\label{E:GInvFourier}
	\cf_{s}^{-1}G(\xi_s) := \int_{-\infty}^{\infty} G(s)e^{2\pi i s \xi_s} \, ds
	= \begin{dcases}
		0 & \xi_s \ge 0\\ \\
		4\pi^2\xi_s e^{4\pi\xi_s A} & \xi_s < 0.\\ 
		\end{dcases}
\end{align}

As a consequence of equations \eqref{E:Plancherel} and \eqref{E:FourierConvolution}, we are able to reduce the dimension of the integral in calculation of the $L^2$-norm of $\bm{L}_{\beta}f(r_z,s_z,t_z)$.  Starting from \eqref{E:ReParam1stConvolution}, the inverse Fourier transform in the $s$ variable yields a statement equivalent to \eqref{E:U2InvFTransform}: 
\begin{align}\label{E:InvFourier_s(L)}
\cf_s^{-1}\bm{L}_{\beta}f(r_z,\xi_s,t_z) = \begin{dcases}
		0 & \xi_s \ge 0\\ \\
		-4\, \xi_s \int_{\R^2} \cf_s^{-1} f(r_{\zeta},\xi_s,t_{\zeta})\, e^{4\pi\xi_s A} \, dr_{\zeta} \,dt_{\zeta} & \xi_s < 0.\\ 
		\end{dcases}
\end{align}

\subsection{A second inverse Fourier transform and Hilbert-Schmidt operators}\label{SS:HilbertSchmidtOperators}

The payoff to the re-parametrization in the previous subsection comes from the fact that $A = \eqref{D:DefLabelA}$ is in convolution form in the $r$ variables.  Motivated by \eqref{E:GInvFourier}, for $\xi_s<0$ we define
\begin{equation}\label{D:Define_H(r)}
H(r) := -4 \,\xi_s \,e^{2\pi\xi_s[(1+\beta)r^2+(1-\beta)(t_z-t_{\zeta})^2+2ir(t_z+t_{\zeta}+\beta(t_z-t_{\zeta}))]}.
\end{equation}
Now take the inverse Fourier transform.  Using the scaling $r \mapsto \sqrt{-2\xi_s(1+\beta)}\,r$ and integral \eqref{E:FourierGaussian}, a computation shows that for $\xi_s<0$,
\begin{align}
\cf_r^{-1} H(\xi_r) &= -4\, \xi_s\,\int_{-\infty}^{\infty}  e^{2\pi i r\,\xi_r} e^{2\pi\xi_s[(1+\beta)r^2+(1-\beta)(t_z-t_{\zeta})^2+2ir(t_z+t_{\zeta}+\beta(t_z-t_{\zeta}))]}\, dr \notag \\ 
&= \frac{2\sqrt{2}\sqrt{-\xi_s}}{\sqrt{1+\beta}}\, \text{Exp}\bigg[\frac{\pi}{2(1+\beta)\xi_s}\big(\xi_r^2 + 4(1+\beta)(\xi_s\xi_rt_z +  2\xi_s^2 t_z^2) \label{E:DefInvFourierH}  \notag \\ 
& \qquad \qquad \qquad \qquad \qquad \qquad \qquad \qquad + 4(1-\beta)(\xi_s\xi_rt_{\zeta} +  2\xi_s^2 t_{\zeta}^2)\big) \bigg] \notag \\
&= \frac{2\sqrt{2}\sqrt{-\xi_s}}{\sqrt{1+\beta}}\, \text{Exp}\bigg[\frac{\pi}{2(1+\beta)\xi_s}\bigg\{\xi_r^2 + 8(1+\beta)\xi_s^2\left(t_z+\frac{\xi_r}{4\xi_s}\right)^2 - \frac{(1+\beta)\xi_r^2}{2} \notag \\ 
& \qquad \qquad \qquad \qquad \qquad \qquad \qquad \qquad + 8(1-\beta)\xi_s^2\left(t_{\zeta}+\frac{\xi_r}{4\xi_s}\right)^2 - \frac{(1-\beta)\xi_r^2}{2}\bigg\} \bigg] \notag \\
&=  \frac{2\sqrt{2}\sqrt{-\xi_s}}{\sqrt{1+\beta}}\, \text{Exp}\left[4\pi\xi_s\left(t_z+\frac{\xi_r}{4\xi_s}\right)^2 \right] \cdot \text{Exp}\left[\frac{4\pi\xi_s(1-\beta)}{1+\beta}\left(t_{\zeta}+\frac{\xi_r}{4\xi_s}\right)^2 \right]. 
\end{align}

The key observation from \eqref{E:DefInvFourierH} is that the $t_z$ and $t_{\zeta}$ variables are decoupled, i.e., $\cf_r^{-1} H(\xi_r)$ 
breaks into a product of functions of these respective variables.  Now define 
\begin{align}\label{E:DefM_0}
m_{0,\xi_r,\xi_s}(t_z) :=  m_0(t_z)
&= \frac{2\sqrt{2}\sqrt{-\xi_s}}{\sqrt{1+\beta}}\, \text{Exp}\left[4\pi\xi_s\left(t_z+\frac{\xi_r}{4\xi_s}\right)^2 \right] \cdot \mathds{1}_{\{\xi_s <\, 0\}}
\end{align} 
and
\begin{align}\label{E:DefM_1}
m_{1,\xi_r,\xi_s}(t_{\zeta}) := m_1(t_{\zeta}) &= \text{Exp}\left[\frac{4\pi(1-\beta)\xi_s}{1+\beta}\left(t_{\zeta} + \frac{\xi_r}{4\xi_s}\right)^2 \right] \cdot \mathds{1}_{\{\xi_s < \,0\}},
\end{align}
where $\mathds{1}_{\{\xi_s < \, 0\}}$ is the indicator function of the interval $(-\infty,0)$ in the $\xi_s$ variable.  These definitions were set up so that for each $\xi_s<0$,
\begin{equation}\label{E:InvFourier_c(H)_Product}
\cf_r^{-1} H(\xi_r) = m_0(t_z)\, m_1(t_{\zeta}).
\end{equation}
Now \eqref{E:InvFourier_s(L)}, \eqref{D:Define_H(r)} and \eqref{E:InvFourier_c(H)_Product} give that

\begin{align}
\cf_r^{-1}\cf_s^{-1}\bm{L}_{\beta}f(\xi_r,\xi_s,t_z) = \begin{dcases}
		0 & \xi_s \ge 0\\ \\
		m_0(t_z) \int_{-\infty}^{\infty} m_1(t_{\zeta}) \cf_r^{-1}\cf_s^{-1} f(\xi_r,\xi_s,t_{\zeta}) \,dt_{\zeta} & \xi_s < 0.\\ 
		\end{dcases}
\end{align}

\bigskip

We finish our analysis by studying the family of operators $\sm_{\xi_r,\xi_s} := \sm$ defined by
\begin{equation}\label{E:HibertSchmidtOpM}
\sm g(t_z) := m_0(t_z) \int_{-\infty}^{\infty} m_1(t_{\zeta}) g(t_{\zeta})\,dt_{\zeta}.
\end{equation}

\begin{proposition}\label{P:MHilbertSchmidtNorm}
For each $\xi_r\in\R$, $\xi_s<0$, the operator $\sm_{\xi_r,\xi_s} = \sm$ is rank-one Hilbert-Schmidt with norm 
\begin{equation*}
\norm{\sm}_{HS} = \frac{1}{\sqrt[4]{1-\beta^2}}.
\end{equation*}
\end{proposition}
\begin{proof}
Observe 
\begin{align*}
\norm{\sm}_{HS}^2 &= \int_{\R^2} |m_0(t_z)\, m_1(t_{\zeta})|^2\,dt_{z}dt_{\zeta} \\
&= \frac{-8\,\xi_s}{1+\beta}\, \left(\int_{-\infty}^{\infty}  \text{Exp}\left[8\pi\xi_s\left(t_z+\frac{\xi_r}{4\xi_s}\right)^2\right] \,dt_{z}\right) \\
& \qquad \qquad   \cdot \left(\int_{-\infty}^{\infty} \text{Exp}\left[\frac{8\pi(1-\beta)\xi_s}{1+\beta}\left(t_{\zeta} + \frac{\xi_r}{4\xi_s}\right)^2 \right] \,dt_{\zeta}\right)   \\
&= \frac{-8\,\xi_s}{1+\beta} \left( \frac{1}{2\sqrt{2}\sqrt{-\xi_s}}\right) \cdot \left(\frac{1}{2\sqrt{2}\sqrt{-\xi_s}}\sqrt{\frac{1+\beta}{1-\beta}} \right) \\
&=\frac{1}{\sqrt{1-\beta^2}}.
\end{align*}
Taking square roots, we are done.
\end{proof}

We are ready to prove Theorem \ref{T:ModelNorm}, but first consider the action of $\sm_{\xi_r,\xi_s} = \sm$ on a function $g_{\xi_r,\xi_s} = g \in L^2(\R)$ for a fixed pair $\xi_r \in \R,\, \xi_s <0$. Recall that the Hilbert-Schmidt norm of an operator dominates its operator norm.  Indeed, 
\begin{align}
\norm{\sm g}_{L^2(\R)}^2 &= \int_{-\infty}^{\infty} \left| m_0(t_z) \int_{-\infty}^{\infty} m_1(t_{\zeta}) g(t_{\zeta})\,dt_{\zeta} \right|^2 \, dt_{z} \notag \\
&\le \left(\int_{-\infty}^{\infty}\left|m_0(t_z)\right|^2\,dt_z\right) \left( \int_{-\infty}^{\infty}\left|m_1(t_{\zeta})\right|^2\,dt_{\zeta}\right) \left( \int_{-\infty}^{\infty}|g(t_{\zeta})|^2\, dt_{\zeta}\right) \label{L2NormOfMg} \\
&= \frac{1}{\sqrt{1-\beta^2}} \norm{g}_{L^2(\R)}^2, \notag
\end{align} 
with equality in \eqref{L2NormOfMg} holding if and only if $g$ is a multiple of $m_1$, i.e., 
\begin{equation}\label{NormMAchieved}
g_{\xi_r,\xi_s}(\cdot) = \varphi(\xi_r,\xi_s)m_{1,\xi_r,\xi_s}(\cdot).
\end{equation}

\subsection{Proof of Theorem \ref{T:ModelNorm}}\label{SS:MainProof}
The reparametrization in Section \ref{SS:ReparametrizingTheKernel} from \eqref{E:LerayModelHypersurface4.1} through \eqref{E:LerayTransAutoReParam} shows that $\norm{\bm{L}_{\beta}f}_{L^2(\cs_{\beta},\sigma)} = \norm{\bm{L}_{\beta}f(r_z,s_z,t_z)}_{L^2(\R^3,\sigma)}$.  By repeated application of Plancherel's identity \eqref{E:Plancherel}, 
\begin{align}
\norm{\bm{L}_{\beta}f(r_z,s_z,t_z)}_{L^2(\R^3,\sigma)} &= \norm{\cf_s^{-1}\bm{L}_{\beta}f(r_z,\xi_s,t_z)}_{L^2(\R^3,\sigma)} \notag \\ 
&= \norm{  \cf_r^{-1}\cf_s^{-1}\bm{L}_{\beta}f(\xi_r,\xi_s,t_z)  }_{L^2(\R^3,\sigma)} \notag \\
&= \norm{  \sm_{\xi_r,\xi_s}\left( \cf_r^{-1}\cf_s^{-1}f(\xi_r,\xi_s,\cdot)\right)(t_z) }_{L^2(\R^3,\sigma)} \notag \\
&\le \frac{1}{\sqrt[4]{1-\beta^2}} \norm{\cf_r^{-1}\cf_s^{-1}f(\xi_r,\xi_s,t_{z})}_{L^2(\R^3,\sigma)} \label{E:ModelNormIneq} \\
&= \frac{1}{\sqrt[4]{1-\beta^2}} \norm{f(r_z,s_z,t_z)}_{L^2(\R^3,\sigma)}. \notag
\end{align}
Noting that $\norm{f(r_z,s_z,t_z)}_{L^2(\R^3,\sigma)} = \norm{f}_{L^2(\cs_{\beta},\sigma)}$ shows $\frac{1}{\sqrt[4]{1-\beta^2}}$ is an upper bound.  

The norm of $\bm{L}_{\beta}$ is achieved when equality holds in \eqref{E:ModelNormIneq}.  By \eqref{NormMAchieved}, this happens if and only if we choose $f \in L^2(\cs_{\beta},\sigma)$ so that 
\begin{equation}\label{E:AchievingNormOfL_beta}
\cf_r^{-1}\cf_s^{-1}f(\xi_r,\xi_s,\cdot) = \varphi(\xi_r,\xi_s)m_{1,\xi_r,\xi_s}(\cdot).
\end{equation}
The square-integrablility of $f$ ensures $\varphi$ must satisfy
\begin{equation}
\int_{-\infty}^{\infty}\int_{-\infty}^{\,0} \frac{1}{\sqrt{-\xi_s}}|\varphi(\xi_r,\xi_s)|^2 \,d\xi_s\,d\xi_r < \infty.
\end{equation}
This completes the proof and establishes item (c) on the list given in Section \ref{SS:LerayTransform}. \qed

\subsection{The Leray transform of $\cs_{\beta}$ is a projection operator}

We establish item (d) from Section \ref{SS:LerayTransform} by showing $\bm{L}_{\beta} \circ \bm{L}_{\beta} = \bm{L}_{\beta}$. 

\begin{proposition}\label{P:MIsAProjection}
For each $\xi_r\in\R$, $\xi_s<0$, the operator $\sm_{\xi_r,\xi_s} = \sm$ defined by equation \eqref{E:HibertSchmidtOpM} is a projection.
\end{proposition}
\begin{proof}
Let $g\in L^2(\R)$, and $m_0$ and $m_1$ be given by equations \eqref{E:DefM_0} and \eqref{E:DefM_1}, respectively.
\begin{align}
\left(\sm \circ \sm\right)(g)(t_{\eta}) &= m_0(t_{\eta})\int_{-\infty}^{\infty} m_1(t_z)\left( m_0(t_z) \int_{-\infty}^{\infty} m_1(t_{\zeta}) g(t_{\zeta})\,dt_{\zeta} \right) \, dt_z \notag \\
&= m_0(t_{\eta}) \int_{-\infty}^{\infty} m_1(t_{\zeta}) g(t_{\zeta})\,dt_{\zeta} \int_{-\infty}^{\infty} m_1(t_z) m_0(t_z)  \, dt_z. \label{MProjectionStep1}
\end{align}
A computation shows that 
\begin{align*}
\int_{-\infty}^{\infty} m_1(t_z) m_0(t_z)  \, dt_z &= \frac{2\sqrt{2}\sqrt{-\xi_s}}{\sqrt{1+\beta}}\,\int_{-\infty}^{\infty} \text{Exp}\left[\frac{ 8\pi\xi_s}{1+\beta} \left(t_z+\frac{\xi_r}{4\xi_s}\right)^2 \right]\,dt_z  \\
&= 1,
\end{align*}
and thus \eqref{MProjectionStep1} $= \sm g(t_{\eta})$.
\end{proof}

\begin{corollary}\label{C:LerayIsAProjection}
$\bm{L}_{\beta}$ is a projection operator from $L^2(\cs_{\beta},\sigma) \to L^2(\cs_{\beta},\sigma)$.
\end{corollary}
\begin{proof}
After Proposition \ref{P:MIsAProjection}, this amounts to symbol pushing
%\sm_{\xi_r,\xi_s}
\begin{align}
\bm{L}_{\beta} \circ \bm{L}_{\beta} &= \cf_s \circ \cf_r \circ \cf^{-1}_r \circ \cf^{-1}_s \circ \bm{L}_{\beta} \circ \cf_s \circ \cf_r \circ \cf^{-1}_r \circ \cf^{-1}_s \circ \bm{L}_{\beta} \notag\\
&= \cf_s \circ \cf_r \circ \left(\sm \circ \cf^{-1}_r \circ \cf^{-1}_s  \right) \circ \cf_s \circ \cf_r \circ \left(\sm \circ \cf^{-1}_r \circ \cf^{-1}_s  \right) \notag\\
&= \cf_s \circ \cf_r \circ \sm \circ \sm \circ \cf^{-1}_r \circ \cf^{-1}_s \notag\\
&= \cf_s \circ \cf_r \circ \sm  \circ \cf^{-1}_r \circ \cf^{-1}_s   \notag\\
&= \cf_s \circ \cf_r \circ \cf^{-1}_r \circ \cf^{-1}_s \circ \bm{L}_{\beta} =  \bm{L}_{\beta}. \notag
\end{align}
A deeper analysis of these operators is provided in Section \ref{SS:FactorLBeta}.
\end{proof}

\begin{remark} \label{R:Non-C1} To see that the closure of $S_\beta$ in projective space fails to be a $C^1$ hypersurface when $0<\beta<1$, apply the projective automorphism $z_1=\widetilde z_1/\widetilde z_2, z_2=1/\widetilde z_2$; then the behavior of 
\begin{equation*}
\widetilde{S_\beta}:= \left\{ (\widetilde z_1,\widetilde z_2)\in \C^2\colon -\text{Im}(\widetilde z_2) = \left|\widetilde z_1\right|^2 + \beta\,\re\left(\frac{\overline{\widetilde z_2}}{\widetilde z_2}\,\widetilde z_1^{\,2}\right)\right\}
\end{equation*}
 near the origin captures the behavior of $S_\beta$ at infinity.

Setting $\widetilde z_j=\widetilde x_j+i\widetilde y_j$, the cubic formula can be used to represent $\widetilde y_2$ as a function of 
$\left(\widetilde x_1, \widetilde y_1, \widetilde x_2\right)$. Computing with the formula one can check that $\frac{\partial\widetilde y_2}{\partial\widetilde x_2}\to 0$ along every line through the origin, whereas $\frac{\partial\widetilde y_2}{\partial\widetilde x_2}$ is a non-zero constant along the parabola $\widetilde x_2=\widetilde x_1^2, \widetilde y_2=0$; thus $\frac{\partial\widetilde y_2}{\partial\widetilde x_2}$ is discontinuous at the origin.
$\lozenge$
\end{remark}

\subsection{Higher dimensional hypersurfaces}\label{SS:HypersurfacesInC^n}
The natural generalization of $\cs_\beta$ to higher dimensions are hypersurfaces parametrized by vectors $(\beta_1,\dots,\beta_{n-1})$ as follows. Independently set each $0\le \beta_j <1$ and define
\begin{align}\label{E:HigherDimensionalModels}
\cs_{(\beta_1,\dots,\beta_{n-1})} := \left\{ (\zeta_1,\dots,\zeta_{n})\in \C^n\colon \text{Im}(\zeta_n) = \sum_{j=1}^{n-1} \left( \left|\zeta_j\right|^2 + \beta_j \re(\zeta_j^2)\right)\right\}.
\end{align}
This form is suggested by equation \eqref{E:PositiveDefiniteForm2}, as such hypersurfaces give local projective approximations to any strongly $\C$-convex hypersurface to two orders of tangency.  A hypersurface of this kind can be converted to a bounded hypersurface using a projective automorphism.  However, the point at $\infty$ is always mapped to a point which is less than $C^1$-smooth -- unless $\beta_j = 0$ for all $j$.  With the exception of this special case -- in which \eqref{E:HigherDimensionalModels} is the Heisenberg surface --  the Lanzani-Stein results in \cite{LanSte14} fail to imply the $L^2$-boundedness of the Leray transform because they require $C^{1,1}$ smoothness.

The $L^2$-boundedness of the Leray transform of \eqref{E:HigherDimensionalModels} may be seen, however, by mirroring the arguments in Section \ref{SS:L2ViaSizeEstimates}.  A sketch is now given.  As usual, define the Leray transform $\bm{L}_{(\beta_1,\dots,\beta_{n-1})} = \bm{L}$ by \eqref{E:LerayIntegralFormula}.  $\bm{L}$ resembles \eqref{E:LerayModelHypersurface3.1} and the measure $\partial\rho(\zeta) \wedge \bar \partial\partial\rho(\zeta)^{(n-1)}$ is a constant multiple of
\begin{align*}
\sigma = dx_1 \wedge dy_1 \wedge \dots \wedge dx_{n-1} \wedge dy_{n-1} \wedge dx_n. \notag
\end{align*}
Here the coordinates $(\zeta_1,\dots,\zeta_n)$ have been identified with $(x_1,y_1,\dots,x_n,y_n)$.  Parametrizing $\cs_{(\beta_1,\dots,\beta_{n-1})}$ by $\R^{2n+1}$ in this way gives the Leray transform as a one-variable convolution	as in \eqref{E:LerayRealConvolution}.  Apply the Fourier transform in this variable and use the triangle inequality to avoid the oscillatory pieces of the integral \`{a} la \eqref{E:TriangleInqeq1}.   This facilitates use of the Fourier transform in the remaining variables much like \eqref{E:2DFourierTrans}.  Following this blueprint through to \eqref{E:CrudeUpperBound} lets us deduce the following:  The Leray transform is a bounded operator on $L^2\left(\cs_{(\beta_1,\dots,\beta_{n-1})},\sigma\right)$ with
\begin{align*}
\norm{\bm{L}} \le \frac{2^{n-1}}{\sqrt{(1-\beta_1^2)(1-\beta_2^2)...(1-\beta_{n-1}^2)}}.
\end{align*}

\begin{remark}
Computing the exact norm of the Leray transform of \eqref{E:HigherDimensionalModels} in higher dimensions presents more difficulties than when $n=2$.  One may start by studying affine automorphisms of $\cs_{(\beta_1,\dots,\beta_{n-1})}$ analogous those in Section \ref{SS:ParametrizingTheModels}.  When $n=2$, the payoff is seen in Proposition \ref{T:CSTParametrization} when $\cs_{\beta}$ is re-parametrized by a (real) $2$-dimensional abelian group of automorphisms acting on a (real) $1$-dimensional curve in $\cs_\beta$.  The higher dimensional version is a re-parametrization of $\cs_{(\beta_1,\dots,\beta_{n-1})}$ by an $n$-dimensional abelian group of automorphisms acting on a $(n-1)$-dimensional surface in $\cs_{(\beta_1,\dots,\beta_{n-1})}$.  It is unclear whether this approach will produce the exact norm of the Leray transform for $n>2$.
$\lozenge$
\end{remark}

\begin{remark}
The $L^p$-regularity $(p \neq 2)$ of the Leray transform on $\cs_\beta$ and the higher dimensional hypersurfaces in \eqref{E:HigherDimensionalModels} is an interesting open question.  These hypersurfaces are less than $C^{1}$-smooth at $\infty$, so the Lanzani-Stein machinery \cite{LanSte14} is unable to imply any $L^p$-regularity.  In \cite{LanSte17b}, the same authors provide explicit examples of hypersurfaces which are $C^m$-smooth, $1<m<2$, but for which to the $L^p$-regularity of the Leray transform fails for {\em every} $1<p<\infty$.  One might hope that a Calder\'{o}n-Zygmund style approach to this problem will yield $L^p$ results, but at this point the details are unclear and likely to involve non-trivial modifications.  Nonetheless, this problem is worthy of further pursuit.
$\lozenge$
\end{remark}

%%%  DO WE WANT TO SAY SOMETHING ABOUT THIS HERE??? 

%\begin{remark}
%{\bf Mention the kernel of $\bm{L}$ and how the norm is achieved iff $f$ is conjugate of dual $CR$-function.}
%\end{remark}

%%%%%%%%%%%%%%%%%%%%%%%%%%%%%%%%%%%%%%%%%%%%%%%%%%%%%%%%%%%%%%%%%%%%%%%%%%%%%%%%%%%%%%%%%%%%%%%%%%%%%%%%%%%%%%%%%%%%%%%%%%%%%%%%%%%%%%%%%%%%%%%%%%%%%%%%%%%%%%%%%%%%%%%%%%%%%%%%%%%%%%%%%%%%%%%%%%%%%%%%%%%%%%%%%%%%%%%%%%%%%%%%%%%%%%%%%%%%%%%%%%%%%%%%%%%%%%%%%%%%%%%%%%%%%%%%%%%%%%%%%%%%%%%%%%%%%%%%%%%%%%%%%%%%%%%%%%%%%%%%%%%%%%%%%%%%%%%%%%%%%%%%%%%%%%%%%%%%%%%%%%%%%%%%%%%%%%%%%%%%%%%%%%%%%%%%%%%%%%%%%%%%%%%%%%%%%%%%%%%%%%%%%%%%%%%%%%%%%%%%%%%%%%%%%%%%%%%%%%%%%%%%%%%%%%%%%%%%%%%%%%%%%%%%%%%%%%%%%%%%%%%%%%%%%%%%%%%%%%%%%%%%%%%%%%%%%%%%%%%%%%%%%%%%%%%%%%%%%%%%%%%%%%%%%%%%%%%%%%%%%%%%%%%% SECTION 5: PROJECTIVE DUAL CR STRUCTURE %%%%%%%%%%%%%%%%%%%%%%%%%%%%%%%%%%%%%%%%%%%%%%%%%%%%%%%%%%%%%%%%

\section{Projective dual $CR$ structures}\label{S:DualCR}

%\com{Restricting  to $\C^2$ so that  CR structures are given by single vector fields.}

In this section we reinterpret the Leray transform $\bm{L}_{\cs}$ with the use of projective dual coordinates and  the projective dual $CR$ structure on a general strongly $\C$-convex hypersurface $\cs$.  The dual coordinates depend, in our presention, on the choice of a matrix $M$, but the dual $CR$ structure will be {\em independent} of that choice.  (This follows from Lemma \ref{L:ProjRel} below.)

Let
\begin{equation}\label{M-Spec}
M=\begin{pmatrix}
c_1  & a_1  &  a_2 \\
b_1  &  m_{11} &  m_{12} \\
b_2  & m_{21}  &  m_{22} 
\end{pmatrix}
\end{equation}
be an invertible 3-by-3 complex matrix.  Use $M$ to define a map $\Phi_M: \C^2\times\C^2\to\C$ by
\begin{align*}
((z_1,z_2),(w_1,w_2) )&\mapsto
\begin{pmatrix} 1 & w_1 & w_2\end{pmatrix} M \begin{pmatrix} 1 \\ z_1 \\ z_2\end{pmatrix}\\
&=c_1+a_1z_1+a_2z_2+b_1w_1+b_2w_2 +\begin{pmatrix}w_1 & w_2\end{pmatrix} 
\begin{pmatrix}
  m_{11} &  m_{12} \\
m_{21}  &  m_{22} 
\end{pmatrix}\begin{pmatrix}  z_1 \\ z_2\end{pmatrix}.
\end{align*}

Given a smooth real hypersurface $\cs\subset\C^2$ with defining function $\rho$ along with $\zeta\in \cs$, let 
\begin{subequations}\label{E:MuDef}
\begin{align}
\mu_1(\zeta)&=\frac{\dee\rho}{\dee \zeta_2}(\zeta)\\
\mu_2(\zeta)&=-\frac{\dee\rho}{\dee \zeta_1}(\zeta)
\end{align}
\end{subequations}
and note that
\begin{equation*}%\label{E:LDef}
L:= \mu_1(\zeta) \frac{\dee}{\dee \zeta_1} + \mu_2(\zeta) \frac{\dee}{\dee \zeta_2}
\end{equation*}
 is a non-vanishing type-(1,0) vector field tangent to $\cs$.  The (affine) complex tangent line for $\cs$ at $\zeta$ may be described parametrically by a map from $\C \to \C^2$ sending
\begin{align} \label{E:TgtLinePar}
\upsilon \mapsto \zeta+\upsilon(\mu_1(\zeta),\mu_2(\zeta)),
\end{align}
or equivalently by
\begin{equation}\label{E:TgtLineEqFromRho}
\left\{(z_1,z_2)\colon \frac{\dee\rho}{\dee \zeta_1}(\zeta)(\zeta_1-z_1)+\frac{\dee\rho}{\dee \zeta_2}(\zeta)(\zeta_2-z_2)=0\right\}.
\end{equation}

\begin{definition}\label{D:MregDef}
We say the hypersurface $\cs$  is $M$-admissible if for all $\zeta\in\cs$,
\begin{equation}\label{E:MRegDef}
\det\begin{pmatrix}
    m_{11}\mu_1(\zeta)+m_{12}\mu_2(\zeta)  & m_{11}\zeta_1+m_{12}\zeta_2+b_1 \\
       m_{21}\mu_1(\zeta)+m_{22}\mu_2(\zeta)  &  m_{21}\zeta_1+m_{22}\zeta_2+b_2
\end{pmatrix} \ne 0
\end{equation}   
\end{definition}

The motivation behind this definition comes from the following lemma.

\begin{lemma}
If $\cs$ is $M$-admissible then there are uniquely-determined functions $w_{1,M}$ and $w_{2,M}$ on $\cs$ with the property that the  complex tangent line to $\cs$ at $\zeta$ is given by 
\begin{equation}\label{E:TgtLineEq}
\left\{(z_1,z_2)\colon\Phi_M\big( (z_1,z_2), (w_{1,M}(\zeta),w_{2,M}(\zeta)) \big)=0\right\}.
\end{equation}
\end{lemma}

\begin{proof} The set \eqref{E:TgtLineEq} is either a complex line or the empty set (corresponding to the projective ``line at infinity").
To prove the lemma, it suffices to check via \eqref{E:TgtLinePar} that there are uniquely-determined $w_{1,M}(\zeta),w_{2,M}(\zeta)$  so that the set \eqref{E:TgtLineEq} contains both $\zeta$ and $\zeta+(\mu_1(\zeta),\mu_2(\zeta))$; this is equivalent to the system
\begin{align*}
(m_{11}\zeta_1+m_{12}\zeta_2+b_1)w_{1,M}(\zeta)+(m_{21}\zeta_1+m_{22}\zeta_2+b_2)w_{2,M}(\zeta)&=-c_1-a_1\zeta_1-a_2\zeta_2\\
(m_{11}\mu_1(\zeta)+m_{12}\mu_2(\zeta))w_{1,M}(\zeta)+(m_{21}\mu_1(\zeta)+m_{22}\mu_2(\zeta))w_{2,M}(\zeta)&=-a_1\mu_1(\zeta)-a_2\mu_2(\zeta).
\end{align*}
Condition \eqref{E:MRegDef} guarantees that this system is uniquely solvable for $w_{1,M}(\zeta),w_{2,M}(\zeta)$.
\end{proof}

We note that since $\zeta$ belongs to the tangent line \eqref{E:TgtLineEq}
we have
\begin{equation}\label{E:ZetaInTgtLine}
\Phi_M\big( (\zeta_1,\zeta_2), (w_{1,M}(\zeta),w_{2,M}(\zeta)) \big)=0.
\end{equation}

\subsection{Examples of $M$-admissibility}\label{SS:ExamplesOfMAdmissibility}

\begin{example}
Setting $$M_1=\begin{pmatrix}
-1 & 0  &  0 \\
0  &  1 &  0 \\
0  & 0  &  1
\end{pmatrix}$$ we find that $\cs$ is $M_1$-admissible if and only if $\zeta_1\mu_2(\zeta)\ne \zeta_2\mu_1(\zeta)$ on $\cs$; using \eqref{E:TgtLinePar} we find that this is equivalent to the condition that no complex tangent line passes through the origin.

Every compact strongly $\C$-convex $\cs$ enclosing the origin is $M_1$-admissible.  See \cite{BarGru18} and Section 2.5 in \cite{ScandBook04} for further discussion.

Note also that in this case \eqref{E:TgtLineEq} reads as 
\begin{equation}\label{E:NormalizeBdd1}
z_1w_{1,M_1}(\zeta)+z_2w_{2,M_1}(\zeta)=1
\end{equation}
for $z$ in the complex tangent line.  Similarly \eqref{E:ZetaInTgtLine} becomes
\begin{equation}\label{E:NormalizeBdd2}
\zeta_1w_{1,M_1}(\zeta)+\zeta_2w_{2,M_1}(\zeta)=1.
\end{equation}
Comparing \eqref{E:NormalizeBdd1} to \eqref{E:TgtLineEqFromRho} we see that we may (and must) take
\begin{align*}
w_{1,M_1}(\zeta)&=\frac{\frac{\dee\rho}{\dee \zeta_1}(\zeta)}{\zeta_1\frac{\dee\rho}{\dee \zeta_1}(\zeta)+\zeta_2\frac{\dee\rho}{\dee \zeta_2}(\zeta)}\\
w_{2,M_1}(\zeta)&=\frac{\frac{\dee\rho}{\dee \zeta_2}(\zeta)}{\zeta_1\frac{\dee\rho}{\dee \zeta_1}(\zeta)+\zeta_2\frac{\dee\rho}{\dee \zeta_2}(\zeta)},
\end{align*}
where the non-vanishing of the denominators follows from the $M_1$-admissibility. It follows now that
\begin{equation*}
\left< \partial\rho(\zeta),(\zeta-z) \right> = \left( \zeta_1\frac{\dee\rho}{\dee \zeta_1}(\zeta)+\zeta_2\frac{\dee\rho}{\dee \zeta_2}(\zeta) \right) \left(1-z_1w_{1,M_1}(\zeta)-z_2w_{2,M_1}(\zeta)\right),
\end{equation*}
where the left hand side appears in the denominator of equation \eqref{E:LerayKernel}.  Thus, the Leray kernel 
\begin{equation*}
\mathscr{L}_{\cs}(z,\zeta)=\frac{1}{(2\pi i)^2} \frac{\partial\rho(\zeta) \wedge \bar \partial\partial\rho(\zeta)}{\left< \partial\rho(\zeta),(\zeta - z) \right>^2}
\end{equation*}
may be written as
\begin{equation*}
\frac{1}{(2\pi i)^2} \frac{ \left(\zeta_1\frac{\dee\rho}{\dee \zeta_1}(\zeta)+\zeta_2\frac{\dee\rho}{\dee \zeta_2}(\zeta)\right)^{-2}\,\partial\rho(\zeta) \wedge \bar \partial\partial\rho(\zeta)}{\left(1-z_1w_{1,M_1}(\zeta)-z_2w_{2,M_1}(\zeta)\right)^2}.
\end{equation*}
To rewrite the numerator further we note that 
\begin{equation*}
\left(\zeta_1\frac{\dee\rho}{\dee \zeta_1}(\zeta)+\zeta_2\frac{\dee\rho}{\dee \zeta_2}(\zeta)\right)^{-1}\,\dee\rho(\zeta)=w_{1,M_1}\,d\zeta_1+w_{2,M_1}\,d\zeta_2,
\end{equation*}
from which we may deduce that
\begin{align*}
\left(\zeta_1\frac{\dee\rho}{\dee \zeta_1}(\zeta)+\zeta_2\frac{\dee\rho}{\dee \zeta_2}(\zeta)\right)^{-2}\,&\partial\rho(\zeta) \wedge \bar \partial\partial\rho(\zeta)\\
&= \left(w_{1,M_1}\,d\zeta_1+w_{2,M_1}\,d\zeta_2 \right)\w\left(\deebar w_{1,M_1}\w d\zeta_1+\deebar w_{2,M_1}\w d\zeta_2 \right)\\
&= \left(w_{2,M_1}\,\deebar w_{1,M_1}-w_{1,M_1}\,\deebar w_{2,M_1}\right)\w d\zeta_1\w d\zeta_2\\
&= \left(w_{2,M_1}\,dw_{1,M_1}-w_{1,M_1}\,dw_{2,M_1}\right)\w d\zeta_1\w d\zeta_2,
\end{align*}
so that 
\begin{equation}\label{E:LerM1}
\mathscr{L}_{\cs}(z,\zeta)
= \frac{1}{(2\pi i)^2}\frac{ \left(w_{2,M_1}\,dw_{1,M_1}-w_{1,M_1}\,dw_{2,M_1}\right)\w d\zeta_1\w d\zeta_2}{\left(1-z_1w_{1,M_1}-z_2w_{2,M_1}\right)^2}.
\end{equation}

In the special case when $\cs$ is the unit sphere in $\C^2$, we have
\begin{align}
w_{1,M_1}(\zeta)&=\overline \zeta_1 \notag\\
w_{2,M_1}(\zeta)&=\overline \zeta_2 \notag\\
\mathscr{L}_{\cs}(z,\zeta)
&= \frac{1}{(2\pi i)^2} \frac{ \left(\overline \zeta_2\,d\overline \zeta_1-\overline \zeta_1\,d\overline \zeta_2\right)\w d\zeta_1\w d\zeta_2}{\left(1-z_1\overline \zeta_1-z_2\overline \zeta_2\right)^2}. \label{E:SzegoSphere}
\end{align}

We note that this formula coincides with the Szeg\H{o} kernel for the unit sphere.  \qed

\end{example}

\begin{example}\label{Ex:M2Example}
Setting $$M_2=\begin{pmatrix}
0 & 0  &  i \\
0  & 2
 &  0 \\
-i  & 0  &  0
\end{pmatrix}$$
we find that $\cs$ is $M_2$-admissible if and only if $\mu_1(\zeta)=\frac{\dee\rho}{\dee \zeta_2}(\zeta)\ne 0$ on $\cs$; equivalently, $\cs$ has no ``vertical" complex tangents.  Every $\cs$ arising as a graph of a smooth function over $\C\times\R$ is $M_2$-admissible.

In this case \eqref{E:TgtLineEq} reads as 
\begin{equation}\label{E:NormalizeGraph}
2  z_1w_{1,M_2}(\zeta)+iz_2-iw_{2,M_2}(\zeta)=0
\end{equation}
for $z$ in the complex tangent line.  Similarly, \eqref{E:ZetaInTgtLine} becomes
\begin{equation}\label{E:NormalizeGraph2}
2  \zeta_1w_{1,M_2}(\zeta)+i\zeta_2-iw_{2,M_2}(\zeta)=0.
\end{equation}

Comparing \eqref{E:NormalizeGraph} to \eqref{E:TgtLineEqFromRho} we find
\begin{subequations}\label{E:WForM2}
\begin{align}
w_{1,M_2}(\zeta)&=\frac{i}{2}\frac{\frac{\dee\rho}{\dee \zeta_1}(\zeta)}{\frac{\dee\rho}{\dee \zeta_2}(\zeta)}, \\
w_{2,M_2}(\zeta)&=\frac{\zeta_1\frac{\dee\rho}{\dee \zeta_1}(\zeta)+\zeta_2\frac{\dee\rho}{\dee \zeta_2}(\zeta)}{\frac{\dee\rho}{\dee \zeta_2}(\zeta)},
\end{align}
\end{subequations}
and
\begin{equation*}
\left< \partial\rho(\zeta),(\zeta-z) \right> 
= i\frac{\dee\rho}{\dee \zeta_2}(\zeta)\cdot
\left(  2  z_1w_{1,M_2}(\zeta)+iz_2-iw_{2,M_2}(\zeta) \right).
\end{equation*}
Thus,
\begin{equation*}
\frac{\partial\rho(\zeta) \wedge \bar \partial\partial\rho(\zeta)}{\left< \partial\rho(\zeta),(\zeta - z) \right>^2}= \frac{ -\left(\frac{\dee\rho}{\dee \zeta_2}(\zeta)\right)^{-2}\,\partial\rho(\zeta) \wedge \bar \partial\partial\rho(\zeta)}{\left(2  z_1w_{1,M_2}(\zeta)+iz_2-iw_{2,M_2}(\zeta)\right)^2}.
\end{equation*}
But note that
\begin{equation*}
\left(\frac{\dee\rho}{\dee \zeta_2}(\zeta)\right)^{-1}\,\partial\rho(\zeta) 
=\frac{2}{i}w_{1,M_2}\,d\zeta_1+d\zeta_2,
\end{equation*}
from which we obtain
\begin{align*}
\left(\frac{\dee\rho}{\dee \zeta_2}(\zeta)\right)^{-2}\,\partial\rho(\zeta) \wedge \bar \partial\partial\rho(\zeta)
&=\left( \frac{2}{i}w_{1,M_2}\,d\zeta_1+d\zeta_2\right)
\w \frac{2}{i} \, \deebar w_{1,M_2}\w d\zeta_1\\
&= -2i \,d w_{1,M_2}\w d\zeta_1\w d\zeta_2,
\end{align*}
and so
\begin{equation}\label{E:LerM2}
\mathscr{L}_{\cs}(z,\zeta)
= \frac{1}{2\pi^2i}\frac{ d w_{1,M_2}\w d\zeta_1\w d\zeta_2 }{ \left(2  z_1w_{1,M_2}(\zeta)+iz_2-iw_{2,M_2}(\zeta)\right)^2 }.
\end{equation}

In the special case when $\cs = \cs_{\beta}$ we have
\begin{align}
w_{1,M_2}(\zeta)&=\overline \zeta_1+\beta\zeta_1 \label{D:W_(1,M_2)} \\
w_{2,M_2}(\zeta)&=\frac{2}{i}\zeta_1(\overline \zeta_1+\beta\zeta_1)+\zeta_2 \label{D:W_(2,M_2)} \\
\mathscr{L}_{\beta}(z,\zeta)
&= \frac{1}{2\pi^2i} \frac{ d \overline\zeta_1\w d\zeta_1\w d\zeta_2 }{ \left(2  z_1w_{1,M_2}(\zeta)+iz_2-iw_{2,M_2}(\zeta)\right)^2 } \notag\\
&= \frac{1}{8\pi^2 i}\frac{d\zeta_2 \wedge d\bar\zeta_1 \wedge d\zeta_1}{\big((\bar\zeta_1 +\beta\zeta_1)(\zeta_1 - z_1)+ \frac i2(\zeta_2-z_2)\big)^2}, \label{E:LerayKernelS_Beta}
\end{align}
which recovers the form of the Leray kernel given in \eqref{E:LerayModelHypersurface3.1}.  When $\beta=0$, this becomes
\begin{align}
w_{1,M_2}(\zeta)&=\overline \zeta_1 \notag\\
w_{2,M_2}(\zeta)&=\frac{2}{i}\zeta_1\overline \zeta_1+\zeta_2=\overline\zeta_2 \notag\\
\mathscr{L}_{0}(z,\zeta)
&= \frac{1}{2\pi^2i} \frac{ d \overline\zeta_1\w d\zeta_1\w d\zeta_2 }{ \left(2  z_1\overline \zeta_1+iz_2-i\overline\zeta_2\right)^2 } \notag\\
&=\frac{1}{\pi ^2}\frac{ dx_1\w dy_1\w dx_2 }{ \left(2  z_1\overline \zeta_1+iz_2-i\overline\zeta_2\right)^2 }. \label{E:Szego_S_0}
\end{align}
Since this is conjugate-$CR$ with respect to $\zeta$ we find that $\mathscr{L}_{0}$ is the Szeg\H{o} kernel of $\cs_0$ with respect to the measure $dx_1\w dy_1\w dx_2$.  (See Chapter 10 in \cite{ChenShawBook}.)  \qed

%\com{EXPAND ON THIS LAST SENTENCE}

\end{example}

\begin{remark}
The non-zero entries in $M_2$ are tailored specifically for use with $\cs_{\beta}$. $\lozenge$
\end{remark}

\begin{remark}\label{R:M2-M3}
Setting 
\begin{equation*}
M_3=\begin{pmatrix}
0 & i  &  0 \\
-i  & 0 &  0 \\
0  & 0  &  2
\end{pmatrix},
\end{equation*}
(or any such matrix with non-zero entries in the same slots) we find that $\cs$ is $M_3$-admissible if and only if $\cs$ has no ``horizontal" complex tangents.  In particular, each point in $\cs$ has a neighborhood that is $M_2$-admissible or $M_3$-admissible. $\lozenge$
\end{remark}

\begin{remark}\label{R:DomainIndependenceLeray}
The formulas \eqref{E:LerM1} and \eqref{E:LerM2} reveal that the Leray transform does have some form of the domain independence property from Section \ref{S:Background} if we allow the use of the dual variables.  Any two $M$-admissible hypersurfaces have identical Leray kernels in the variables $(z,w) = (z_1,z_2,w_{1,M}, w_{2,M})$.  The $w$ variables, of course, are hypersurface dependent. $\lozenge$
\end{remark}

\subsection{Universal dual coordinate description of the Leray transform}

Recall that a projective automorphism is a (partially-defined) map from $\C^2$ to $\C^2$ extending to an automorphism of projective space.  These have the form 
\begin{equation}\label{E:ProjTrans}
\Upsilon^*\colon (w_1,w_2)\mapsto\left( \frac{ (1,w_1,w_2) \Upsilon_1}{ (1,w_1,w_2) \Upsilon_0},\frac{ (1,w_1,w_2) \Upsilon_2}{ (1,w_1,w_2) \Upsilon_0}\right),
\end{equation}
where $\Upsilon_0,\Upsilon_1,\Upsilon_2$ are the columns of an invertible 3-by-3 matrix $\Upsilon$.

\begin{lemma}\label{L:ProjRel}
If $\cs$ is both $M$-admissible and $M'$-admissible, then there is a projective automorphism  $\Upsilon^*$  so that $(w_{1,M'},w_{2,M'})=\Upsilon^*(w_{1,M},w_{2,M})$.
\end{lemma}

%\com{commas in row vectors}

\begin{proof}
The hypothesis guarantees that each complex tangent line for $\cs$ may be written uniquely
as the zero set of 
$%\begin{equation*}
\left( 1, w_{1,M}, w_{2,M}\right) M \left( 1, z_1, z_2\right)^T
$%\end{equation*}
or of
$%\begin{equation*}
\left( 1, w_{1,M'}, w_{2,M'}\right) M' \left( 1, z_1, z_2\right)^T
$; %\end{equation*} 
 thus there is a relation of the form 
$%\begin{equation*}
\left( 1, w_{1,M'}, w_{2,M'}\right) M' 
= \kappa(w_{1,M},w_{2,M}) \cdot \left( 1, w_{1,M}, w_{2,M}\right) M
$ %\end{equation*}
or
\begin{equation}\label{E:W-Transf}
\begin{pmatrix} 1, w_{1,M'}, w_{2,M'}\end{pmatrix}  
= \kappa(w_{1,M},w_{2,M}) \cdot \begin{pmatrix} 1, w_{1,M}, w_{2,M}\end{pmatrix} M (M')^{-1}.
\end{equation}
  A computation shows that this works if we let $\Upsilon=M (M')^{-1}$ and set
   $\kappa(w_{1,M},w_{2,M})=\frac{ 1}{ (1,w_{1,M},w_{2,M}) \Upsilon_0}$.  The non-vanishing of the denominator follows from \eqref{E:W-Transf}.
\end{proof}

We now obtain the universal dual coordinate description of the Leray transform.

\begin{proposition}\label{P:GenMLeray}
If $\cs$ is an $M$-admissible strongly $\C$-convex hypersurface in $\C^2$ then the Leray integral from \eqref{E:LerayIntegralFormula} may be written as
\begin{equation}\label{E:GenMLeray}
\bm{L}_{\cs}f(z) =  \int_{\zeta\in \cs} f(\zeta)
\frac{\nu_M}
{\big((1, w_{1,M}, w_{2,M}) M (1, z_1, z_2)^T\big)^2},
\end{equation}
where
\begin{multline}\label{E:NuDef}
\nu_M:= \frac{1}{(2\pi i)^2}\Bigg(\left(  d w_{1,M}, d w_{2,M}\right) 
\begin{pmatrix} m_{11} & m_{12} \\ m_{21} & m_{22}\end{pmatrix}
\begin{pmatrix} a_2 \\ -a_1\end{pmatrix}\\
+
\begin{vmatrix} m_{11} & m_{12} \\ m_{21} & m_{22}\end{vmatrix}
\left( w_{2,M}\,d w_{1,M} - w_{1,M}\,d w_{2,M}\right)\Bigg)\w d\zeta_1\w d\zeta_2.
\end{multline}
Here $a_j$ and $ m_{jk}$ are given by \eqref{M-Spec}.
\end{proposition}

Note that from the examples in Section \ref{SS:ExamplesOfMAdmissibility}, we already have this result for the special matrices $M_1$ and $M_2$.  In particular,
\begin{equation}\label{E:Nu-M2}
\nu_{M_2}=\frac{1}{2\pi^2i} \,d w_{1,M_2}\w d\zeta_1\w d\zeta_2.
\end{equation}

\begin{proof}
The result is local.  Suppose that $\cs$ is $M_2$-admissible at a particular point.  The transformation laws from the proof of Proposition \ref{L:ProjRel} (with $M'=M_2$)  yield 
\begin{equation*}
\nu_M=\frac{1}{2\pi^2i}\big( (1,w_{1,M},w_{2,M}) \Upsilon_0 \big)^2 \, d\left( \frac{ (1,w_{1,M},w_{2,M}) \Upsilon_1}{ (1,w_{1,M},w_{2,M}) \Upsilon_0} \right)\w d\zeta_1\w d\zeta_2, 
\end{equation*}
with \[\Upsilon=\begin{pmatrix}
-i a_2  & \frac{a_1}{2}  &  i c_1 \\
-i m_{12}  & \frac{m_{11}}{2}  & i b_1  \\
-i m_{22}   &  \frac{m_{21}}{2}  &  i b_2 
\end{pmatrix}.\]  Routine computation then reduces $\nu_M$ to the form given in \eqref{E:NuDef}.  If $S$ fails to be $M_2$-admissible at some point then by Remark \ref{R:M2-M3} it will be $M_3$-admissible there and a similar computation will yield the result. 
\end{proof}

\begin{remark} \label{R:NuTransA}
Similarly, in the setting of Lemma \ref{L:ProjRel} we have the transformation law
\begin{equation*}
\nu_M=\big( (1,w_{1,M},w_{2,M}) \Upsilon_0 \big)^2 \,\nu_{M'}
\end{equation*}
with $(1,w_{1,M},w_{2,M}) \Upsilon_0$ non-vanishing.  $\lozenge$
\end{remark}

\begin{proposition}\label{P:NuNonDegen}
If $\cs$ is an $M$-admissible strongly $\C$-convex hypersurface in $\C^2$ then the form $\nu_M$ from \eqref{E:NuDef} is nowhere-vanishing as a 3-form on $\cs$.
\end{proposition}

\begin{proof}
In view of Remarks \ref{R:M2-M3} and \ref{R:NuTransA} it suffices to consider the case $M=M_2$ (or the similar case $M=M_3$).  The claim is equivalent to the non-vanishing of $d\rho\w \nu_{M_2}$ along $\cs$.

Recalling \eqref{E:Nu-M2} and \eqref{E:WForM2} (and using subscripts to denote derivatives) we have
\begin{align}
d\rho\w \nu_{M_2}
&= \frac{1}{4\pi^2}\left(\rho_{\bar\zeta_1}\,d\bar\zeta_1 + \rho_{\bar\zeta_2}\,d\bar\zeta_2\right)\w
d\left(  \frac{\rho_{\zeta_1}}{\rho_{\zeta_2}}\right)\w
d\zeta_1\w d\zeta_2 \notag\\
&= \frac{1}{4\pi^2}\rho_{\zeta_2}^{-2}\left(\rho_{\bar\zeta_1}\,d\bar\zeta_1 + \rho_{\bar\zeta_2}\,d\bar\zeta_2\right)\notag\\
&\qquad
\w\Big(\rho_{\zeta_2}\left(\rho_{\zeta_1\bar\zeta_1}\,d\bar\zeta_1+ \rho_{\zeta_1\bar\zeta_2}\,d\bar\zeta_2 \right)
- \rho_{\zeta_1}\left(\rho_{\zeta_2\bar\zeta_1}\,d\bar\zeta_1+ \rho_{\zeta_2\bar\zeta_2}\,d\bar\zeta_2\right)     
\Big)\w d\zeta_1\w d\zeta_2 \notag\\
&=\frac{1}{4\pi^2}\rho_{\zeta_2}^{-2} \det\begin{pmatrix}
 0 &   \rho_{\bar \zeta_1} & \rho_{\bar \zeta_2}  \\
\rho_{ \zeta_1}  & \rho_{\zeta_1 \bar \zeta_1}  & \rho_{\zeta_1 \bar \zeta_2}  \\
\rho_{ \zeta_2}  &   \rho_{\zeta_2 \bar \zeta_1}  & \rho_{\zeta_2 \bar \zeta_2}   
\end{pmatrix}\,d\zeta_1\w d\zeta_2\w d\bar\zeta_1\w d\bar\zeta_2. \label{E:PhaseFunction}
\end{align}
Since $\cs$ is strongly $\C$-convex, it is also strongly pseudoconvex -- see Remark \ref{R:StrongCConvex => StrongPseudoconvex} -- and strong pseudoconvexity is well-known to be equivalent to the negativity of the determinant above. This establishes the claim.
\end{proof}

\begin{remark}\label{R:NuTransB}
In view of Proposition \ref{P:NuNonDegen}, we write $\nu_M$ in the form $\ph{\nu_M} \cdot|\nu_M|$ where $\ph{\nu_M}$ is a unimodular scalar function on $\cs$ (the {\em phase function} for $\nu_M$) and $|\nu_M|$ is a positive 3-form on $\cs$.  (We may view $|\nu_M|$  as a measure on $\cs$ that is a smooth positive multiple of the surface area measure.)

From Remark \ref{R:NuTransA} we have
\begin{align*}
|\nu_M|&=\big| (1,w_{1,M},w_{2,M}) \Upsilon_0 \big|^2 \,|\nu_{M'}|,\\
\ph{\nu_M} &= \frac{\big( (1,w_{1,M},w_{2,M}) \Upsilon_0 \big)^2 }{\big| (1,w_{1,M},w_{2,M}) 
\Upsilon_0 \big|^2}  \ph{\nu_{M'}}.
\end{align*}

For future reference we note that if $\cs$ is $M_2$-admissible with defining function $\rho$ that is independent of $\re(\zeta_2)$ then $\rho_{ \zeta_2}$ is purely imaginary and non-vanishing and thus $\cs$ may be written locally in the form $\text{Im} (\zeta_2) = \lambda(\zeta_1, \bar\zeta_1)$ (and thus $\cs$ is {\em (locally) rigid} in the sense of Baouendi, Rothschild and Tr\`eves -- see \cite{BaoRothTre1985}). From \eqref{E:PhaseFunction} we see that in this situation the 4-form $d\rho\w\nu_{M_2} $ is positive along $\cs$; equivalently, $\nu_{M_2}$ is positive as a 3-form on $\cs$\, and hence  $|\nu_M|=\nu_M$, $\ph{\nu_M}=1$. $\lozenge$
\end{remark}

\subsection{Projective dual $CR$-structures}\label{SS:ProjDualCR}
Recall the non-vanishing type-$(1,0)$ tangent vector field
\begin{equation}\label{E:LDef}
L = \mu_1(\zeta) \frac{\dee}{\dee \zeta_1} + \mu_2(\zeta) \frac{\dee}{\dee \zeta_2}.
\end{equation}

\begin{lemma}\label{L:DualDiffeo}
$\overline L w_{1,M}$ and $\overline L w_{2,M}$ do not vanish simultaneously.
\end{lemma}

\begin{proof}
If $\cs$ is $M_2$-admissible then from \eqref{E:Nu-M2} and Proposition \ref{P:NuNonDegen}
we have $d w_{1,M_2}\w d\zeta_1\w d\zeta_2\ne 0$ as a 3-form on $\cs$.  It follows that $d w_{1,M_2}$ fails to be $\C$-linear on the maximal complex subspace at any point of  $\cs$, that is, 
$\overline L w_{1,M_2}\ne 0$.  A similar argument works if $\cs$ is $M_3$-admissible.  The general case follows now from Remark \ref{R:M2-M3} and Lemma \ref{L:ProjRel}.
\end{proof}

\begin{lemma}\label{L:EtaSetup}
There is a uniquely (and locally) determined smooth function $\eta$ on $\cs$ so that if we set
\begin{align*}
L_{\sf dual}&=\overline L - \overline\eta L \\
\overline L_{\sf dual}&= L - \eta \overline L,
\end{align*}
we have $\overline L_{\sf dual} w_{1,M} = 0 = \overline L_{\sf dual} w_{2,M}$ for all $M$ for which $\cs$ is $M$-admissible.
\end{lemma}

\begin{proof}
 At points where $\cs$ is $M_2$-admissible,   Lemma \ref{L:ProjRel} allows us to assume that $M=M_2$. 
 
 From \eqref{E:MuDef} and \eqref{E:WForM2} we obtain $w_{1,M_2}(\zeta)=\frac{\mu_2(\zeta)}{2i\mu_1(\zeta)}$ and hence
 \begin{equation}\label{E:Eq1}
2i(L\zeta_1)w_{1,M_2}=\,L\zeta_2.
\end{equation}
Applying $L$ to \eqref{E:NormalizeGraph2} and using \eqref{E:Eq1}, we find that
\begin{subequations}\label{E:Eq2}
\begin{align}
 0 &= 2(L\zeta_1)w_{1,M_2}+2 \zeta_1 (Lw_{1,M_2})-i Lw_{2,M_2}+ i L\zeta_2 \notag \\
    &=2 \zeta_1 (Lw_{1,M_2})-i Lw_{2,M_2}  \\
 0 &= 2 \zeta_1 (\overline{L}w_{1,M_2})-i\overline{L}w_{2,M_2}. \label{E:Eq2b}
\end{align}
\end{subequations}

By the proof of  Lemma \ref{L:DualDiffeo}  we have $\overline{L}w_{1,M_2}\ne0$, allowing us to define 
\begin{equation}\label{E:EtaDef}
\eta= \frac{L w_{1,M_2}}{\overline{L} w_{1,M_2}},
\end{equation}
so that $\overline L_{\sf dual} w_{1,M_2} = 0$.  From (\ref{E:Eq2}a), (\ref{E:Eq2}b) we see that also $\overline L_{\sf dual} w_{2,M_2} = 0$ as required.

A similar argument holds in the $M_3$-admissible case, and the general case follows as before by application of Remark \ref{R:M2-M3}. \end{proof}

We now define the (projective) dual {\em CR}-structure on $\cs$:

\begin{definition}\label{D:DualCRStructure}
We declare $L_{\sf dual}$ to be dual-type $(1,0)$; thus a function $f$ on $\cs$ is dual-$CR$ if and only if $\overline L_{\sf dual} f=0$.  Note that this definition is set up so that $w_{1,M}$ and $w_{2,M}$ are dual-$CR$ for all choices of $M$ for which $\cs$ is $M$-admissible.
\end{definition} 

Referring now to Proposition \ref{P:GenMLeray} (or to \eqref{E:LerM2} or \eqref{E:LerM1}) we see that $\bm{L}_{\cs}f(z)$ is obtained by integrating
$f$ against the dual-$CR$ function  $ \big((1, w_{1,M}, w_{2,M}) M (1, z_1, z_2)^T\big)^{-2}$ with respect to the 3-form $\nu_M$.  This may be compared with the corresponding Szeg\H o projection, obtained by integrating against a conjugate-$CR$ function.

\section{Factorization of $\bm{L}_{S}$}\label{S:FactorLS}

For any $M$-admissible strongly $\C$-convex hypersurface $\cs$, the three-form $\nu_M$, the measure $|\nu_M|$ and the phase function $\ph{\nu_M}$ are defined in \eqref{E:NuDef} and Remark \ref{R:NuTransB}.

In the special case of $\cs_{\beta}$ we have $|\nu_{M_2}|=\nu_{M_2}=\pi^{-2}\,dx_1\w dy_1\w dx_2$
and $\ph{\nu_{M_2}}=1$ (with notation as in Section \ref{SS:L2ViaSizeEstimates}).  In fact, from Remark \ref{R:NuTransB} we have $\ph{\nu_{M_2}}=1$ for all rigid hypersurfaces, and the proof of Proposition \ref{P:NuNonDegen} (see equation \eqref{E:PhaseFunction}) can be used to check that $\ph{\nu_{M_2}}$ is constant only for hypersurfaces that are $\zeta_2$-rotations of rigid hypersurfaces.

For the remainder of this section, all $L^2$-norms are defined using the measure $|\nu_M|$ where  $M$ is a matrix for which $\cs$ is $M$-admissible.  The dependence of $|\nu_M|$ (and $\ph{\nu_M}$) on the choice of $M$ was explained in Remark \ref{R:NuTransB}.

\subsection{Orthogonal and skew projections}\label{SS:OrthSkewProj}

Focus now on the case of smooth bounded strongly $\C$-convex $\cs$ (which will be $M_1$-admissible after a translation).
Let $\Omega$ denote the domain bounded by $\cs$.  Lanzani and Stein's main result in \cite{LanSte14} guarantees that $\bm{L}_{\cs}$ defines a bounded projection operator from $L^2(\cs,|\nu_M|)$ onto the Hardy space $H^2(\cs,|\nu_M|)$ of $L^2$ boundary values of holomorphic functions on $\Omega$.  We often omit the measure when writing these spaces. In particular, we have
that $\bm{L}_{\cs} \circ \bm{L}_{\cs}=\bm{L}_{\cs}$ and 
\begin{equation*}
H^2(\cs)=\ker\left(\bm{L}_{\cs}-\bm{I}\right)=\bm{L}_{\cs}\left(L^2(\cs)\right).
\end{equation*}

Let $\bm{P}_{\cs}$ denote the orthogonal projection from $L^2(\cs)$ onto $\left(\ker \bm{L}_{\cs}\right)^\perp$. Define the space $W_{M}(\cs)=W(\cs)$ by
\begin{equation}\label{E:W(S)definition}
W(\cs)=\left\{ \ph{\nu_M} h : h \in \left(\ker \bm{L}_{\cs}\right)^\perp \right\},
\end{equation}
and let $\bm{R}_{\cs}\colon L^2(\cs)\to \chd(\cs)$ denote the surjective operator given by $f\mapsto\ph{\nu_M}\cdot\bm{P}_{\cs}f$.  Note that  $\|\bm{R}_{\cs}\|=\|{\bm{P}}_{\cs}\|=1$. 
For each $z\in \cs$, define the dual-$CR$ function
\begin{align}
g_z(\zeta) &=\Phi_M\big( (z_1,z_2), (w_{1,M}(\zeta),w_{2,M}(\zeta)) \big)^{-2} \notag\\ 
&=\big((1, w_{1,M}(\zeta), w_{2,M}(\zeta)) M (1, z_1, z_2)^T\big)^{-2}.
\end{align}

\begin{lemma}\label{L:gzDensity} The conjugate  space $\overline{W(\cs)}$ is the closed span of $\left\{g_z\colon z\in\Omega\right\}$. 
\end{lemma}

\begin{proof}
From Proposition \ref{P:GenMLeray}, $\bm{L}_{\cs}f(z) = \int_{\cs} f \nu_M g_z = \int_{\cs} f \left|\nu_M\right| \ph{\nu_M} g_z$. Hence, $\bm{L}_{\cs}f=0$ if and only if 
\begin{equation*}
f\perp_{L^2(\cs,|\nu_M|)}\left\{\overline{\ph{\nu_M} g_z}\colon z\in\Omega\right\}.
\end{equation*}
Consequently,
$\left(\ker \bm{L}_{\cs}\right)^\perp$ is the closed span of $\left\{\overline{\ph{\nu_M} g_z}\colon z\in\Omega\right\}$ and thus $W(\cs)$ is the closed span of $\left\{\overline{g_z}\colon z\in\Omega\right\}$.  The claim follows.
\end{proof}

\begin{remark}
We will show below in  Proposition \ref{P:ConjDualHardy} that $\overline{W(\cs)}$ is in fact the Hardy space $H^2_\dual(\cs)$ corresponding to the projective dual $CR$ structure.  $\lozenge$
\end{remark}

For $f\in L^2(\cs), h\in (\ker \bm{L}_{\cs})^\perp$ we have $\overline{\ph{\nu_M} h}\in\overline{W(\cs)}$ and thus
\begin{align*}
\int_\cs f\,\nu_M\,\overline{\ph{\nu_M} h} = \int_\cs f \,|\nu_M| \, \overline{h} &= \int_\cs \left(\bm{P}_{\cs}f\right) |\nu_M| \,\overline{h}\\
&=\int_\cs \left(\bm{R}_{\cs}f\right) |\nu_M| \, \overline{\ph{\nu_M} h}.
\end{align*}

Recalling Proposition \ref{P:GenMLeray} and setting $\overline{\ph{\nu_M} h}=g_z$ above, we find that
\begin{align}\label{E:FactorLThruR}
\bm{L}_{\cs} f(z) = \int_\cs f \,\nu_M \, g_z =  \int_\cs \left(\bm{R}_{\cs}f\right)  |\nu_M| \, g_z &=  \int_\cs \overline{\ph{\nu_M}} \left(\bm{R}_{\cs}f\right) \nu_M \, g_z\notag\\
&= \bm{L}_{\cs} \left(  \overline{\ph{\nu_M}}\left(\bm{R}_{\cs}f\right)\right) (z).
\end{align}

\begin{theorem}\label{T:FactorLeray}
Define the operator $\bm{Q}_{\cs}:\chd(\cs)\to H^2(\cs)$ by 
$
 f\mapsto\bm{L}_{\cs}\left(\overline{\ph{\nu_M}}\cdot f\right).
$
Then
\refstepcounter{equation}\label{N:Q-Op}
\begin{enum}
\item $\bm{L}_{\cs} = \bm{Q}_{\cs} \circ \bm{R}_{\cs}$ \label{I:LQR}
\item $\|\bm{L}_{\cs}\|=\|\bm{Q}_{\cs}\|$\label{I:LNormFromQNorm}
\item $\bm{Q}_{\cs}$ is invertible. \label{I:QInv}
\end{enum}
\end{theorem}

\begin{proof}
The factorization \itemref{N:Q-Op}{I:LQR} follows from \eqref{E:FactorLThruR}.

To prove \itemref{N:Q-Op}{I:LNormFromQNorm} note that $\|\bm{Q}_{\cs}\|\le \|\bm{L}_{\cs}\|$ from the definition of $\bm{Q}_{\cs}$ and that $\|\bm{L}_{\cs}\|\le \|\bm{Q}_{\cs}\|$ from \itemref{N:Q-Op}{I:LQR} and $\|\bm{R}_{\cs}\|=1$.

To prove \itemref{N:Q-Op}{I:QInv} note first that surjectivity of $\bm{Q}_{\cs}$ follows from \itemref{N:Q-Op}{I:LQR}.  To verify injectivity, note that $f=\ph{\nu_M}h\in \chd(\cs)$, $\bm{Q}_{\cs} f=0$ implies
$h\in \left(\ker \bm{L}_{\cs}\right)^\perp$ and 
$\bm{L}_{\cs}h=0$ hence $h=0=f$.
\end{proof}  

Suppose now that $\cs$ is the $M$-admissible strongly $\C$-convex boundary of an {\em unbounded} domain $\Omega$.  As before, equip $\cs$ with the measure $|\nu_M|$.  Define $\bm{L}_{\cs}, g_z$ etc. as above, and suppose that the following hold:
\refstepcounter{equation}\label{N:UnbddSCond}
\begin{enum}
\item $\|g_z\|_{ L^2(\cs)}$ is a locally bounded function of  $z\in\Omega$, hence $\bm{L}_{\cs}$ maps $L^2(\cs) \to \co(\Omega)$  ;
\item Taking boundary values, we obtain a projection operator $\bm{L}_{\cs}\colon L^2(\cs)\to L^2(\cs)$.
\end{enum}

If we then define the Hardy space $H^2(\cs)\subset L^2(\cs)$ to be the range of this operator we find that Theorem \ref{T:FactorLeray} and the preceding discussion carries over immediately.  When $\cs = \cs_\beta$, compare (\ref{N:UnbddSCond}a) with Corollary \eqref{C:UnifNormBddsOnCompacts}.

\subsection{Factorization of $\bm{L}_{\beta}$}\label{SS:FactorLBeta}

Restrict focus now to $\cs_\beta$ with the measure $\sigma = dx_1 \wedge dy_1 \wedge dx_2$, though we often omit $\sigma$ below.  Recall that $\cs_{\beta}$ is $M_2$-admissible and  observe that $\sigma = \pi^2|\nu_{M_2}|$.  From Section \ref{SS:HilbertSchmidtOperators} we have the Paley-Wiener-type result that  \[\reallywidehat{H^2(\cs_\beta)}:=\cf_r^{-1}\cf_s^{-1}\left(H^2(\cs_\beta)\right)\] is the set of square-integrable functions
$h(\xi_r,\xi_s,t) = \varphi(\xi_r,\xi_s)m_{0,\xi_r,\xi_s}(t)$. Recalling that $m_{0,\xi_r,\xi_s}(t)$ vanishes for $\xi_s \ge 0$,
\begin{align}\label{E:FT-of-HardyNormSq}
\norm{h}_{L^2(\R^3)}^2 &= \int_{-\infty}^\infty\int_{-\infty}^0\int_{-\infty}^\infty |\varphi(\xi_r,\xi_s)m_{0,\xi_r,\xi_s}(t)|^2\,dt\,d\xi_s\,d\xi_r \notag \\
&= \frac{2\sqrt{2}}{1+\beta} \int_{-\infty}^\infty\int_{-\infty}^0 \sqrt{-\xi_s}\cdot|\varphi(\xi_r,\xi_s)|^2\,d\xi_s\,d\xi_r<\infty.
\end{align}

Recalling the general formula $\reallywidehat{\overline f}(\xi)=\overline{\reallywidehat f}(-\xi)$ and noting that  $m_{1,\xi_r,\xi_s}(t)=\overline{m_{1,\xi_r,\xi_s}(t)}$ we similarly find that the spaces from Section \ref{SS:OrthSkewProj} correspond to 
\begin{subequations} 
\begin{align}
\hat{L^2(\cs_\beta)} &= L^2(\R^3),\\
\reallywidehat{\ker \bm{L}_{\beta}} &= \left\{\psi\in L^2(\R^3)\colon \int_\R \psi(\xi_r,\xi_s,t)m_{1,\xi_r,\xi_s}(t)\,dt=0 \text{ a.e. }(\xi_r,\xi_s)\in\R^2\right\},\\
\reallywidehat{\left(\ker \bm{L}_{\beta}\right)^\perp} &= \reallywidehat{W(\cs_{\beta})} \\
&= \Bigg\{\varphi(\xi_r,\xi_s)m_{1,\xi_r,\xi_s}(t)\colon 
 \int_{-\infty}^{\infty} \int_{-\infty}^{\,0} \frac{1}{\sqrt{-\xi_s}} |\varphi(\xi_r,\xi_s)|^2\,d\xi_s\,d\xi_r<\infty\Bigg\}, \notag\\
\reallywidehat{\overline{W(\cs_{\beta}})} &= \Bigg\{\varphi(\xi_r,\xi_s)m_{1,-\xi_r,-\xi_s}(t)\colon  \int_{-\infty}^{\infty} \int_{0}^{\infty} \frac{1}{\sqrt{\xi_s}} |\varphi(\xi_r,\xi_s)|^2\,d\xi_s\,d\xi_r < \infty\Bigg\}. \label{E:FourierDualHardySpace}
\end{align}
\end{subequations}
Similarly, the operators correspond to
\begin{subequations} 
\begin{align}
\hat{\bm{L}_{\beta}} &= \sm\\
\hat{\bm{R}_{\beta}}&=\hat{\bm{P}_{\beta}}\colon \psi \mapsto \frac{m_{1,\xi_r,\xi_s}(t)}{\|m_{1,\xi_r,\xi_s}(\cdot)\|_{L^2(\R)}^2}
\int_\R \psi(\xi_r,\xi_s,u)m_{1,\xi_r,\xi_s}(u)\,du\\
\hat{\bm{Q}_{\beta}} &= \sm\Big|_{\widehat{W(\cs_{\beta})}}
\end{align}
\end{subequations}
respectively.  The following diagram keeps track of these spaces and operators.

%%%%%%%%%%%%%%%%%%%%%%%%%%%%%%%%%
\begin{center}
\begin{tikzpicture}
  \matrix (m) [matrix of math nodes,row sep=4em,column sep=4em,minimum width=2em]
  {
     L^2(\cs_{\beta}) & \ & \overset{\cf_r^{-1}\cf_s^{-1}}{\parbox{1.0cm}{\rightarrowfill}}  & L^2(\R^3) & \  \\
     W(\cs_{\beta}) & H^2(\cs_\beta) \ & \overset{\cf_s\cf_r}{\parbox{1.0cm}{\leftarrowfill}} & \reallywidehat{W(\cs_{\beta})} & \reallywidehat{H^2(\cs_{\beta})}  \\};
     
%\hat{\left(\ker \bm{L}_{\beta}\right)^\perp} =

  \path[-stealth]
  
    (m-1-1) edge node [left] {$\bm{P}_{\beta}=\bm{R}_{\beta}$} (m-2-1)
            edge node [above] {$\ \ \ \ \ \bm{L}_{\beta}$}
            (m-2-2)
            
    (m-2-1) edge node [below] {$\bm{Q}_{\beta}$}
            (m-2-2)

    (m-1-4) edge node [left] {$\hat{\bm{R}_{\beta}}$} (m-2-4) %\hat{\bm{P}_{\beta}}=
            edge node [above] {$\ \ \ \ \ \ \ \ \hat{\bm{L}_{\beta}} = \sm$} (m-2-5)
            
    (m-2-4) edge node [below] {$\hat{\bm{Q}_{\beta}}$}
            (m-2-5);
            
  %\path[latex-latex]
  
    %(m-1-2) edge node [below] {$\hat{\bm{Q}_{\beta}}$}
    %        (m-1-3);

\end{tikzpicture}
\end{center}
%%%%%%%%%%%%%%%%%%%%%%%%%%%%%%%%

It is easy to use these formulas to provide a separate verification of the assertions of Theorem \ref{T:FactorLeray} adapted to $\cs_\beta$.  In fact, even more is true in this special case.

\begin{theorem}\label{T:ConstTimesIsom}
The operator $\sqrt[4]{1-\beta^2}\,\bm{Q}_{\beta}\colon \chd(\cs_{\beta})\to H^2(\cs_\beta)$ is an isometry.
\end{theorem}
This follows immediately from equation \eqref{E:AchievingNormOfL_beta}.

%%%%%%%%%%%%%%%%%%%%%%%%%%%%%%%%%%%%%%%%%%%%%%% SECTION 6 %%%%%%%%%%%%%%%%%%%%%%%%%%%%%%%%%%%%%%%%%%%%

\section{Dual Hardy spaces} \label{S:DualHardy}

\begin{theorem}\label{T:GlobalDual}
If $\cs\subset\C^2$ is an $M$-admissible compact strongly $\C$-convex hypersurface bounding a domain $\Omega$, then the map $w_M=(w_{1,M},w_{2,M})$ is a diffeomorphism from $\cs$ onto a compact strongly $\C$-convex hypersurface $\cs_\dual$ bounding a domain $\Omega_\dual\subset\C^2$. 
\end{theorem}

Note that $\cs_\dual$ and $\Omega_\dual$ depend on $M$.  Whatever the choice of $M$, the $CR$ functions on a relatively open subset of $\cs_\dual$ pull back with respect to $w_M$ to functions on a relatively open subset of $\cs$ that are dual-$CR$ as defined in \ref{D:DualCRStructure}.

\begin{proof}
This is proved in Proposition 2.5.12 of \cite{ScandBook04}.
\end{proof}

\begin{remark}
It is also true that $\cs_\dual$ is $M^{T}$-admissible and that the map $w_{M^{T}}$ for $\cs_\dual$ is the inverse of the map $w_M$ for $\cs$ (see Section 6 in \cite{Bar16}). $\lozenge$
\end{remark}

\subsection{Pullback operators and function spaces}\label{SS:PullbackOpsFunctions}
Continuing with the assumptions of Theorem \ref{T:GlobalDual}, a function on a relatively open subset of $\cs_\dual$ will be $CR$ if and only if its pullback via $w_M$ is $CR$ with respect to the projective dual $CR$-structure described in Section \ref{SS:ProjDualCR}. The pullback $H^{2}_\dual(\cs):=w_M^*\left( H^2(\cs_\dual)\right)$ is the {\em projective dual Hardy space} for $\cs$.  The functions in $H^{2}_\dual(\cs)$ are $L^2$ boundary values of holomorphic functions on 
$\Omega_\dual$.

\begin{remark}
It should be emphasized that while the norms of functions in $H_\dual^2(\cs)$ depend on our choice of matrix $M$, the functions themselves are independent of this choice. $\lozenge$
\end{remark}

The operator $\bm{L}_{\cs_\dual}$ pulls back to an operator $\bm{L}_{\cs}^\dual\colon L^2(\cs)\to L^2(\cs)$ satisfying 
\begin{equation*}\label{E:ProjRuleDual}
\bm{L}_{\cs}^\dual \circ \bm{L}_{\cs}^\dual =\bm{L}_{\cs}^\dual
\end{equation*}\label{E:HardyDualFromLeray}
and
\begin{equation*}
H^2_\dual(\cs)=\ker\left(\bm{L}_{\cs}^\dual-\bm{I}\right)=\bm{L}_{\cs}^\dual\left(L^2(\cs)\right).
\end{equation*}

\begin{proposition}\label{P:NuAdjoint}
For $f,g\in L^2(\cs)$ we have
\begin{equation*}
\int_\cs \left(\bm{L}_{\cs}f\right)\nu_M\, g = \int_\cs f\,\nu_M \left(\bm{L}_{\cs}^\dual g\right)
\end{equation*}
\end{proposition}

\begin{proof}
This is Theorem 25 in \cite{Bar16}.
\end{proof}

\begin{proposition}\label{P:ConjDualHardy}
The space $W(\cs)$ defined by equation \eqref{E:W(S)definition} satisfies $\overline{W(\cs)}=H^2_\dual(\cs)$.
\end{proposition}

\begin{proof}
We start by noting that
\begin{align*}
f\in\ker\bm{L}_{\cs} 
&\Leftrightarrow \int_\cs (\bm{L}_{\cs}f)\nu_M\, g = 0 \text{ for all } g\in L^2(\cs)\\
& \Leftrightarrow \int_\cs f\,\nu_M \left(\bm{L}_{\cs}^\dual g\right) = 0 \text{ for all } g\in L^2(\cs)
& & \text{[by Prop.\,\ref{P:NuAdjoint}]}\\
& \Leftrightarrow f \ph{\nu_M} \perp \overline{H^2_\dual(\cs)}\\
& \Leftrightarrow f \perp \overline{\ph{\nu_M}H^2_\dual(\cs)}.
\end{align*}
Thus $\left( \ker\bm{L}_{\cs} \right)^\perp = \overline{\ph{\nu_M}H^2_\dual(\cs)}$.

Multiplying both sides by $\ph{\nu_M}$ we have $W(\cs)=\overline{H^2_\dual(\cs)}$, as required.
\end{proof}
Combining this with Lemma \ref{L:gzDensity} we find that $\text{span}\left\{g_z\colon z\in\Omega\right\}$ is dense in $H^2_\dual(\cs)$.  Dualizing, we find it also true that the space $\text{span}\left\{h_w\colon w\in\Omega_\dual\right\}$ is dense in $H^2(\cs)$, where
\begin{equation*}
h_w(\zeta)=\Phi_{M}\big( (\zeta_1,\zeta_2), (w_1,w_2) \big)^{-2}.
\end{equation*}

The diagram below depicts the operators $\bm{P}_{\cs}, \bm{R}_{\cs}, \bm{Q}_{\cs}$ introduced in Section \ref{SS:OrthSkewProj}, when the Leray transform is decomposed into factors.  We now see that this factorization passes through the conjugate dual Hardy space.
%%%%%%%%%%%%%%%%%%%%%%%%%%%%%%%%%
\begin{center}
\begin{tikzpicture}
  \matrix (m) [matrix of math nodes,row sep=5em,column sep=8em,minimum width=2em]
  {
     L^2(\cs) & \ & \ \\
     (\ker \bm{L}_{\cs})^{\perp} & \overline{H^2_\dual(\cs)} & H^2(\cs)\\};
     
  \path[-stealth]
    (m-1-1) edge node [left] {$\bm{P}_{\cs}$} (m-2-1)
            %edge [double] node [below] {$\mathcal{B}_t$} (m-1-2)
            edge node [right] {$\ \ \ \   \bm{L}_{\cs}$} (m-2-3)
            edge node [left] {$\ \ \ \ \ \ \ \bm{R}_{\cs}$} (m-2-2)
            
    %(m-2-1.east|-m-2-2) edge node [below] {$\tau$}
        %    node [above] {$\text{mult by}$} (m-2-2);
    %(m-1-2) edge node [right] {$\mathcal{B}_T$} (m-2-2)
            %edge [dashed,-] (m-2-1);
            
    (m-2-1) edge node [below] {$(f \mapsto \tau(\nu_{M})\cdot f)$}
            (m-2-2)
            
    (m-2-2) edge node [below] {$\bm{Q}_{\cs}$} (m-2-3);       
\end{tikzpicture}
\end{center}
%%%%%%%%%%%%%%%%%%%%%%%%%%%%%%%%

\subsection{The dual hypersurface of $\cs_{\beta}$}\label{SS:DualOfSbeta}
Turning attention now to the $\cs_\beta$, the assumptions of Theorem \ref{T:GlobalDual} do not apply but we have from \eqref{D:W_(1,M_2)} and \eqref{D:W_(2,M_2)} that 
\begin{equation*}
w_{M_2}\colon (\zeta_1,\zeta_2)\mapsto \left(\bar{\zeta_1}+\beta\zeta_1,\frac{2}{i}\zeta_1(\overline \zeta_1+\beta\zeta_1)+\zeta_2\right)
=\left(\bar{\zeta_1}+\beta\zeta_1, \overline{\zeta_2}-i\beta \zeta_1^2+i\beta \overline{\zeta_1}^2   \right)
\end{equation*}
with inverse
\begin{equation*}
(w_1,w_2)\mapsto \left(\frac{\overline{w_1}-\beta w_1}{1-\beta^2},w_2+\frac{2iw_1(\overline{w_1}-\beta w_1)}{1-\beta^2}\right).
\end{equation*}
This leads to
\begin{equation*}
\cs_{\beta,\dual} := \left\{ (w_1,w_2)\in \C^2: -(1-\beta^2)\text{Im}(w_2) = \left|w_1\right|^2 - \beta\,\text{Re}(w_1^2)\right\}.
\end{equation*}

Note that $\cs_{\beta,\dual}$ is linearly equivalent to $\cs_\beta$ via the map $(w_1,w_2)\mapsto \left(\frac{iw_1}{\sqrt{1-\beta^2}},-w_2\right)$.  The space $H^2_\dual(\cs_\beta)$ and operator $\bm{L}_{\cs_\beta}^\dual$ are induced from $H^2(\cs_{\beta,\dual})$ and $\bm{L}_{\cs_{\beta,\dual}}$ as above.  The proof of Proposition \ref{P:ConjDualHardy} carries over to show that we still have $\overline{W(\cs_{\beta})}=H^2_\dual(\cs_{\beta})$.  

We previously discussed the inverse Fourier transforms of both this space and the original $H^2(\cs_{\beta})$ in Section \ref{SS:FactorLBeta}.  For reference, we include them below.  Recall that $m_{0,\xi_r,\xi_s}(t)$ and $m_{1,\xi_r,\xi_s}(t)$ are defined in \eqref{E:DefM_0} and \eqref{E:DefM_1}.
\begin{align}
\reallywidehat{H^2(\cs_{\beta})} &= \left\{\varphi(\xi_r,\xi_s) m_{0,\xi_r,\xi_s}(t): \int_{-\infty}^{\infty}\int_{-\infty}^0 \sqrt{-\xi_s}\cdot|\varphi(\xi_r,\xi_s)|^2 \,d\xi_s d\xi_r < \infty \right\}, \\
\reallywidehat{H_\dual^2(\cs_{\beta})} &= \left\{\varphi(\xi_r,\xi_s) m_{1,-\xi_r,-\xi_s}(t): \int_{-\infty}^{\infty}\int_{0}^\infty \frac{1}{\sqrt{\xi_s}}\cdot|\varphi(\xi_r,\xi_s)|^2 \,d\xi_s d\xi_r < \infty \right\}.
\end{align}

Note that the space $\reallywidehat{H^2(\cs_{\beta})}$ and the {\em conjugate} dual space \reallywidehat{\overline{H_\dual^2(\cs_{\beta})}} coincide when $\beta = 0$.

%%%%%%%%%%%%%%%%%%%%%%%%%%%%%%%%%%%%%%%%%%%%%%% APPENDICIES %%%%%%%%%%%%%%%%%%%%%%%%%%%%%%%%%%%%%%%%%%%%

\appendix

\section{$L^2$-norms of the kernel function} \label{A:LbetaConvegence}

Like the Cauchy transform in one complex variable, the Leray transform on a bounded $\C$-convex hypersurface $\cs$ constructs holomorphic functions on the $\C$-convex domain it bounds.  But dealing with unbounded hypersurfaces requires more delicacy.  This appendix shows that $\bm{L}_{\beta}$ constructs holomorphic functions on $\Omega_{\beta}$ from boundary functions in $L^2(\cs_{\beta},\sigma)$.

\subsection{Automorphisms of $\Omega_{\beta}$}\label{SS:Aut of Omega_{beta}}

Partition $\Omega_{\beta}$ as an infinite union of translates of $\cs_{\beta}$.  For all $\epsilon \ge 0$, define the hypersurface 
\begin{equation*}
\cs_{\beta}^{\,\epsilon} := \{(z_1,z_2) : (z_1,z_2 - i\epsilon) \in \cs_{\beta} \},
\end{equation*}
and note that if $z \in \cs_{\beta}^{\,\epsilon}$, then $\im(z_2) = |z_1|^2 + \beta \re(z_1^2) + \epsilon$.  It is clear that $\cs_{\beta}^{0} = \cs_{\beta}$ and $\Omega_{\beta} = \cup_{\epsilon > 0} \, \cs_{\beta}^{\,\epsilon}.$  The maps $\phi_{(c,s)}$ defined by equation \eqref{E:DefinitionPhi_(c,s)} extend to automorphisms of $\Omega_{\beta}$.  In fact, they preserve each shell:

\begin{proposition}\label{P:AutOfShells}
The affine map 
\begin{equation}\label{E:AutomorphismDefRedux}
\phi_{(c,s)}(z_1,z_2) = \big(z_1+c, z_2 + 2i(\bar c + \beta c)z_1 + i\big(|c|^2+\beta\, {\mathrm{Re}}(c^2)\big)+ s \big)
\end{equation}
is an automorphism of each $\cs_{\beta}^{\,\epsilon}$ for all choices of $c\in \C$ and $s \in \R$.
\end{proposition}
\begin{proof}
Choose any $z \in \cs_{\beta}^{\,\epsilon}$.  Then $\im(z_2) = |z_1|^2 +\beta\re(z_1^2)+\epsilon$.  Writing the components of $\phi_{(c,s)}(z_1,z_2)$ as $(\phi_1,\phi_2)$, we see
\begin{align*}
|\phi_1|^2 = |z_1|^2 + |c|^2 + 2\re(\bar{z_1}c), \qquad \beta\re(\phi_1^2) = \beta\re(z_1^2) +2\beta\re(z_1c)+\beta\re(c^2),
\end{align*}
and
\begin{align*}
\im(\phi_2) = \im(z_2) + \beta\re(z_1^2) + \epsilon + 2\re(\bar{z_1} c) + 2\beta\re(z_1c) + |c|^2 + \beta\re(c^2).
\end{align*}
This means $\phi_{(c,s)}(z_1,z_2) = (\phi_1,\phi_2)\in \cs_{\beta}^{\,\epsilon}$.
\end{proof}

\begin{corollary}\label{C:SendZtoAxis}
Fix $z\in \cs_{\beta}^{\,\epsilon}$.  There is a unique $\phi_{(c(z),\,s(z))}$ which sends $z=(z_1,z_2) \mapsto (0,i\epsilon)$.
\end{corollary}
\begin{proof}
It can be verified that the choice of
\begin{align}\label{E:C(z)andS(z)def}
c(z) = -z_1, \qquad s(z) = -\re(z_2) - 2\beta\im(z_1^2)
\end{align}
sends $z \mapsto (0,i\epsilon)$.  Uniqueness follows from the form of \eqref{E:AutomorphismDefRedux}.
\end{proof}

%Suppose $\cs \subset \C^n$ is a strongly $\C$-convex hypersurface with defining function $\rho$, and let $\Omega$ be the domain on its $\C$-convex side.  A {\em Leray-Levi measure} on $\cs$ is constant multiple of the $(n,n-1)$-form $\partial\rho(\zeta) \wedge (\bar \partial\partial\rho(\zeta))^{n-1}$.  The Leray kernel $\mathscr{L}_{\cs}(z,\zeta)$ as defined in equation \eqref{E:LerayKernel} is written in terms of a Leray-Levi measure.  A {\em Fefferman hypersurface measure} $\sigma_{\Fef}$ on any strongly pseudoconvex $\cs \subset \C^n$ is defined up to a constant by the following equation:
%\begin{equation}
%\sigma_{\Fef} \wedge d\rho = M(\rho)^{\frac{1}{n+1}}\,dV,
%\end{equation}
%where 
%\begin{equation}
%M(\rho) = -\det \begin{pmatrix}
%    \rho      & \rho_{\bar{z}_k}  \\
%    \rho_{z_j}       & \rho_{z_j \bar{z}_k}  \\
%\end{pmatrix}.
%\end{equation}

%The strong pseudoconvexity of $\cs$ guarantees that $M(\rho)>0$.  Let $\ell_{\cs}(z,\zeta)$ define a function on $\overline{\Omega} \times \cs$ such that there is a Fefferman hypersurface measure $\sigma_{\Fef}$ so that

It was mentioned in Remark \ref{R:AffineTransforms} that the maps $\phi_{(c,s)}$ preserve both volume and the boundary measure $\sigma = dx_1\wedge dy_1 \wedge dx_2$.  The Leray kernel $\mathscr{L}_{\cs}$ is defined in equation \eqref{E:LerayKernel} as an $(n,n-1)$-form, but when $\cs=\cs_{\beta}$ we may think of as $\mathscr{L}_{\beta}(z,\zeta) = \ell_{\beta}(z,\zeta)\sigma(\zeta)$, where $\ell_{\beta}$ is a function (i.e. a $(0,0)$-form) times the measure $\sigma$.  This coefficient function satisfies the following invariance property:

\begin{theorem}\label{T:ellInvariance}
Let $\ell_{\beta}$ denote the coefficient function of $\mathscr{L}_{\beta}$ written with respect to the measure $\sigma$.  Fix a point $z \in \Omega_{\beta}$, choose the unique $\epsilon$ such that $z \in \cs_{\beta}^{\,\epsilon}$, and let $\phi^* = \phi_{(c(z),s(z))}$ denote the map in Corollary \ref{C:SendZtoAxis} sending $z \mapsto (0,i\epsilon)$.  Then
\begin{equation}\label{E:ellInvariance}
\ell_{\beta}(\phi^*(z),\phi^*(\zeta)) = \ell_{\beta}(z,\zeta).
\end{equation}
\end{theorem}
\begin{proof}
Following the parametrization of $\cs_{\beta}$ with the automorphisms $\phi_{(r,s)}$ in Section \ref{SS:ReparametrizingTheKernel}, write the points $z \in \cs_{\beta}^{\,\epsilon}, \, \zeta \in \cs_{\beta}$ as
\begin{align*}
z &= \left(r_z +it_z, \, s_z - 2(1+\beta)r_z t_z + i\left[(1+\beta)r_z^2 +(1-\beta)t_z^2 + \epsilon \right]\right), \\
\zeta &= \left(r_{\zeta} +it_{\zeta}, \, s_{\zeta} - 2(1+\beta)r_{\zeta} t_{\zeta} + i\left[(1+\beta)r_{\zeta}^2 +(1-\beta)t_{\zeta}^2\right]\right).
\end{align*}
Starting from the definition of $\bm{L}_{\beta}f(z)$ in equation \eqref{E:LerayModelHypersurface4.1}, we first note that 
\begin{align*}
d\zeta_2 \wedge d\bar\zeta_1 \wedge d\zeta_1 &= 2i \, ds_{\zeta}\wedge dr_{\zeta}\wedge dt_{\zeta} \\
&= 2i\,\sigma(\zeta).
\end{align*}

Repeating the steps from \eqref{E:DenomPiece1} through \eqref{E:LerayTransAutoReParam} -- with the only difference arising from the fact that now $z\in \cs_{\beta}^{\,\epsilon}$ -- we find that $\bm{L}_{\beta}f(z) = \int_{\cs_{\beta}} f(\zeta)\ell_{\beta}(z,\zeta)\sigma(\zeta)$,
where 
\begin{align}\label{E:ellDef}
\ell_{\beta}(z,\zeta) &= \pi^{-2} \Big[\big((1+\beta)(r_z-r_{\zeta})^2+ (1-\beta)(t_z-t_{\zeta})^2+\epsilon \big) \\
& \qquad \qquad \qquad \qquad+ i\big(2(r_z-r_{\zeta})(t_z+t_{\zeta}+\beta(t_z-t_{\zeta})) - (s_z-s_{\zeta})\big) \Big]^{-2}. \notag
\end{align}

Written with respect to this parametrization, the subscripts of the map $\phi^* =\phi_{(c(z),s(z))}$ defined in \eqref{E:C(z)andS(z)def} take the form
\begin{equation}\label{E:CandSReparametrized}
c(z) = -r_z-it_z, \qquad s(z) = -s_z+2(1-\beta)r_zt_z.
\end{equation}
By construction, $\phi^*(z) = (0,i\epsilon)$.  This is equivalent to saying $\phi^*$ maps 
\begin{equation}\label{E:ZMapsTo}
r_z \mapsto 0, \qquad t_z \mapsto 0, \qquad s_z \mapsto 0.
\end{equation}
Calculating $\phi^*(\zeta)$ is more involved, but starting from \eqref{E:AutomorphismDefRedux} and \eqref{E:CandSReparametrized} it is seen that
\begin{align*}
\phi^*(\zeta) &= \Big((r_{\zeta}-r_z)+i(t_{\zeta}-t_z)\, ,\, s_{\zeta}-s_z + 2(r_z-r_{\zeta})\big((1-\beta)t_z + (1+\beta)t_{\zeta}\big) \\ 
&\qquad \qquad \qquad \qquad \qquad \qquad +i\Big((1+\beta)(r_z-r_{\zeta})^2 + (1-\beta)(t_z-t_{\zeta})^2\Big) \Big). \notag
\end{align*}
This is equivalent to saying that $\phi^*$ maps
\begin{equation}\label{E:ZetaMapsTo}
r_{\zeta} \mapsto r_{\zeta}-r_z, \qquad t_{\zeta} \mapsto t_{\zeta}-t_z,\qquad s_{\zeta} \mapsto s_{\zeta}-s_z + 4t_z(r_z-r_{\zeta}).
\end{equation}

Substituting \eqref{E:ZMapsTo} and \eqref{E:ZetaMapsTo} into \eqref{E:ellDef} shows
\begin{align*}
\ell_{\beta}(\phi^*(z),\phi^*(\zeta)) &= \pi^{-2} \Big[\big((1+\beta)(r_z-r_{\zeta})^2+ (1-\beta)(t_z-t_{\zeta})^2+\epsilon \big) \notag \\
& \qquad \qquad + i\big(2(\beta-1)(r_z-r_{\zeta})(t_z-t_{\zeta}) + s_{\zeta}-s_z +4t_z(r_z-r_{\zeta})\big) \Big]^{-2} \notag \\
&= \pi^{-2} \Big[\big((1+\beta)(r_z-r_{\zeta})^2+ (1-\beta)(t_z-t_{\zeta})^2+\epsilon \big) \\
& \qquad \qquad \qquad \qquad+ i\big(2(r_z-r_{\zeta})(t_z+t_{\zeta}+\beta(t_z-t_{\zeta})) - (s_z-s_{\zeta})\big) \Big]^{-2}.
\end{align*}
This last line equals $\ell_{\beta}(z,\zeta)$.
\end{proof}

Theorem \ref{T:ellInvariance} shows that for fixed $\epsilon > 0$, the $L^2(\cs_{\beta},\sigma)$ norm of $\ell(z,\cdot)$ remains constant as $z$ varies in $\cs_{\beta}^{\,\epsilon}$.  Indeed, the map $\phi^* = \phi_{(c(z),s(z))}$ maps the point $z$ to $(0,i\epsilon)$.  Because the maps $\phi_{(c,s)}$ -- and their inverses, which have the same form -- preserve the measure $\sigma$,
\begin{align}
\int_{\cs_{\beta}} |\ell_{\beta}(z,\zeta)|^2 \sigma(\zeta) &= \int_{\cs_{\beta}} \left|\ell_{\beta}(\phi^*(z),\phi^*(\zeta))\right|^2 \sigma(\zeta) \notag \\
&= \int_{\cs_{\beta}} \left|\ell_{\beta}((0,i\epsilon),\zeta)\right|^2 \sigma(\zeta).  \label{E:l_betaInvariance}
\end{align}

\begin{remark}\label{R:BoltTransformationLaw}
It can be shown that $\sigma$ is a constant multiple of Fefferman hypersurface measure (see \cite{Bar06,Bar16,Fef79,Gup17}) on $\cs_{\beta}$.  Theorem \ref{T:ellInvariance} can be deduced from the general transformation law given in \cite{Bol05} applying to $\mathscr{L}_{\cs}$ written with respect to Fefferman measure.  $\lozenge$
\end{remark}

\begin{remark}\label{R:NonIsoDilation}
For $\alpha>0$, consider the {\em non-isotropic dilation} map
\begin{equation}\label{E:NonIsoDilation}
\delta_{\alpha}(z_1,z_2) = \left(\sqrt{\alpha}\,z_1,\alpha z_2\right).
\end{equation}
It can be checked that these maps are automorphisms of both $\cs_{\beta}$ and $\Omega_{\beta}$.  More generally, $\delta_{\alpha}$ is a bijection from $\cs_{\beta}^{\epsilon} \to \cs_{\beta}^{\alpha\epsilon}$.  These maps no longer preserve volume or the surface measure $\sigma$, but the transformation law in \cite{Bol05} still applies to action of $\delta_{\alpha}$. $\lozenge$

\end{remark}

\subsection{Each $f\in L^2(\cs_{\beta},\sigma)$ generates a holomorphic function $\bm{L}_{\beta}f$ on $\Omega_{\beta}$}\label{SS:L_beta(f)isHolo}

Building on the identity \eqref{E:l_betaInvariance}, we state the following proposition.

\begin{proposition}\label{P:ApdxB.2_Prop}
The integral
\begin{equation*}
\int_{\cs_{\beta}} \left|\ell_{\beta}((0,i\epsilon),\zeta)\right|^2 \sigma(\zeta) = \frac{1}{4\pi^2\epsilon^2\sqrt{1-\beta^2}}.
\end{equation*}
\end{proposition}

The proof of Proposition \ref{P:ApdxB.2_Prop} is split into the following two computational lemmas.

\begin{lemma}\label{L:ApdxB.2_Lem1}
For each $\epsilon>0$,
\begin{equation*}
\int_{\cs_{\beta}} \left|\ell_{\beta}((0,i\epsilon),\zeta)\right|^2 \sigma(\zeta) = \frac{1}{\epsilon^2}\int_{\cs_{\beta}} \left|\ell_{\beta}((0,i),\zeta)\right|^2 \sigma(\zeta).
\end{equation*}
\end{lemma}
\begin{proof}
From \eqref{E:ellDef}, $\ell_{\beta}((0,i\epsilon),\zeta)$ may be written by setting $r_z=t_z=s_z=0$.

\begin{align}
\int_{\cs_{\beta}} \left|\ell_{\beta}((0,i\epsilon),\zeta)\right|^2 \sigma(\zeta) 
&= \frac{1}{\pi^4} \int_{\R^3} \frac{ds_{\zeta}\wedge dr_{\zeta} \wedge dt_{\zeta}}{\left|(1+\beta) r_{\zeta}^2 + (1-\beta)t_{\zeta}^2 + \epsilon + i\left(s_{\zeta}-2(1-\beta)r_{\zeta}t_{\zeta}\right)\right|^4} \notag \\
&= \frac{1}{\epsilon^2\pi^4} \int_{\R^3} \frac{ds_{\zeta}\wedge dr_{\zeta} \wedge dt_{\zeta}}{\left|(1+\beta) r_{\zeta}^2 + (1-\beta)t_{\zeta}^2 + 1 + i\left(s_{\zeta}-2(1-\beta)r_{\zeta}t_{\zeta}\right)\right|^4} \label{E:ApdxB_COV1} \\
&= \frac{1}{\epsilon^2}\int_{\cs_{\beta}} \left|\ell_{\beta}((0,i),\zeta)\right|^2 \sigma(\zeta). \label{E:ApdxB_IntEq}
\end{align}
\end{proof}
\begin{remark}
The change of variable $(r_{\zeta},s_{\zeta},t_{\zeta}) \mapsto (\sqrt{\epsilon}\,r_{\zeta}, \epsilon\,s_{\zeta},\sqrt{\epsilon}\,t_{\zeta} )$ used in \eqref{E:ApdxB_COV1} is the non-isotropic dilation $\delta_{\epsilon}$ in \eqref{E:NonIsoDilation}.  It is also possible to use the transformation law alluded to in Remark \ref{R:NonIsoDilation} to immediately deduce Lemma \ref{L:ApdxB.2_Lem1}. $\lozenge$
\end{remark}

\begin{lemma}\label{L:ApdxB.2_Lem2}
The integral
\begin{equation*}
\int_{\cs_{\beta}} \left|\ell_{\beta}((0,i),\zeta)\right|^2 \sigma(\zeta) = \frac{1}{4\pi^2\sqrt{1-\beta^2}}.
\end{equation*}
\end{lemma}
\begin{proof}
Setting $a(r_{\zeta},t_{\zeta})=(1+\beta)r_{\zeta}^2 + (1-\beta)t_{\zeta}^2 + 1$, equation \eqref{E:ApdxB_IntEq} shows that
\begin{align}
\pi^4 \int_{\cs_{\beta}} \left|\ell_{\beta}((0,i),\zeta)\right|^2 \sigma(\zeta) &= \int_{\R^2} \left[\int_{\R} \left(a(r_{\zeta},t_{\zeta})^2 + \left(s_{\zeta}-2(1-\beta)r_{\zeta}t_{\zeta}\right)^2 \right)^{-2}ds_{\zeta} \right] dr_{\zeta} \wedge dt_{\zeta} \notag \\
&:= \int_{\R^2} I(r_{\zeta},t_{\zeta})\, dr_{\zeta} \wedge dt_{\zeta}. \label{L:I(x1,y1)integral}
\end{align}  
A computation shows that for each $r_{\zeta},t_{\zeta}\in\R$, the quantity in brackets
\begin{align*}
I(r_{\zeta},t_{\zeta}) &=  \frac{\pi}{2\left((1+\beta) r_{\zeta}^2 + (1-\beta)t_{\zeta}^2 + 1\right)^{3}}.
\end{align*}
Making the change of variable $(r_{\zeta},t_{\zeta}) \mapsto \left(\frac{r_{\zeta}}{\sqrt{1+\beta}}, \frac{t_{\zeta}}{\sqrt{1-\beta}}\right)$, we see that 
\begin{align*}
\eqref{L:I(x1,y1)integral} &= \frac{\pi}{2\sqrt{1-\beta^2}} \int_{\R^2} \frac{dr_{\zeta}\wedge dt_{\zeta}}{(r_{\zeta}^2+s_{\zeta}^2 +1)^3} = \frac{\pi^2}{4\sqrt{1-\beta^2}}.
\end{align*}
Dividing by $\pi^4$ gives the result.
\end{proof}

\begin{corollary}\label{C:UnifNormBddsOnCompacts}
The function $\Omega_\beta\to\R$ given by $z\mapsto \int_{\cs_{\beta}} |\ell_{\beta}(z,\zeta)|^2 \sigma(\zeta)$ is uniformly bounded on compact subsets of $\Omega_\beta$.
\end{corollary}

\begin{proof}
Every compact subset of $\Omega_\beta$ is contained in a union of shells $\cup_{\epsilon > \epsilon_0} \, \cs_{\beta}^{\,\epsilon}$ with $\epsilon_0>0$.  The desired conclusion then follows from \eqref{E:l_betaInvariance} and Proposition \ref{P:ApdxB.2_Prop}.
\end{proof}

We are ready to prove the main result of Appendix \ref{A:LbetaConvegence}.  This will verify item (a) from the list in Section \ref{SS:LerayTransform}.

\begin{theorem} $\bm{L}_{\beta}f \in \co(\Omega_{\beta})$ for each $f \in L^2(\cs_{\beta},\sigma)$.
\end{theorem}
\begin{proof}

It will suffice to prove that $\bm{L}_{\beta}f $ is holomorphic on $U$ for  each  relatively compact ball $U\subset\Omega_{\beta}$.

In the special case of compactly-supported $f$ this follows from a standard differentiate-the-integral argument.  

To handle general $f$ we pick a sequence of compactly supported $f_j\in L^2(\cs_{\beta},\sigma)$ with 
$f_j\to f$ in $L^2$.  Then with the use of Corollary \ref{C:UnifNormBddsOnCompacts} we find that
$\bm{L}_{\beta}f_j\to\bm{L}_{\beta}f$ uniformly on $U$ and thus that $\bm{L}_{\beta}f $ is indeed holomorphic on $U$.

%Fix some $\epsilon_0>0$ and choose any $z \in \cup_{\epsilon > \epsilon_0} \, \cs_{\beta}^{\,\epsilon}$.  Combining equation \eqref{E:l_betaInvariance} with Proposition \ref{P:ApdxB.2_Prop}, we have
%\begin{align*}
%\left|\bm{L}_{\beta}f(z) \right| = \left|\int_{\cs_{\beta}} f(\zeta) \ell_{\beta}(z,\zeta) \sigma(\zeta) \right|  &\le \norm{f}_2 \left(\int_{\cs_{\beta}} |\ell_{\beta}(z,\zeta)|^2\sigma(\zeta) \right)^{1/2} \\ &
%\le \frac{\norm{f}_2}{2\pi\epsilon_0 \sqrt[4]{1-\beta^2}}.
%\end{align*}
%
%This bound leads to the conclusion that $\bm{L}_{\beta}f(z)$ is a continuous function and allows for the application of Morera's theorem.  Given any triangle $\Delta \subset \cup_{\epsilon > \epsilon_0} \, \cs_{\beta}^{\,\epsilon}$, we may now interchange the order of integration below:
%\begin{equation}\label{E:Morera}
%\int_\Delta \int_{\cs_{\beta}} f(\zeta) \ell_{\beta}(z,\zeta) \sigma(\zeta) dz = \int_{\cs_{\beta}} f(\zeta) \int_\Delta \ell_{\beta}(z,\zeta) dz\, \sigma(\zeta).
%\end{equation}
%Since $\ell_{\beta}(z,\zeta)$ is holomorphic in $z$, the inner integral on the right hand side of \eqref{E:Morera} vanishes.  Therefore $\bm{L}_{\beta}f(z)$ is holomorphic on the domain $\cup_{\epsilon > \epsilon_0} \, \cs_{\beta}^{\,\epsilon}$.  Send $\epsilon_0 \to 0$ to conclude that $\bm{L}_{\beta}f$ is holomorphic on the full $\Omega_{\beta}$.
\end{proof}

\bibliographystyle{acm}
\bibliography{BarEdh18.bib}

\end{document}